\documentclass[a4paper,abstract=true,DIV=default,toc=index]{scrartcl}
\KOMAoption{overfullrule}{true}
\KOMAoption{twoside}{true}

\usepackage[normalem]{ulem}

\usepackage[hyperpageref]{backref}

\usepackage{amsmath}
\usepackage{amsthm}
\usepackage{amsfonts}
\usepackage{color,enumitem,graphicx}

\usepackage[english]{babel}
\usepackage[dvipsnames]{xcolor}
\usepackage{hyperref}

\newcommand\myshade{85}
\colorlet{mylinkcolor}{violet}
\colorlet{mycitecolor}{YellowOrange}
\colorlet{myurlcolor}{Aquamarine}

\hypersetup{
	linkcolor  = mylinkcolor!\myshade!black,
	citecolor  = mycitecolor!\myshade!black,
	urlcolor   = myurlcolor!\myshade!black,
	colorlinks = true,
}

\usepackage[T1]{fontenc}
\usepackage[libertine,vvarbb]{newtx}
\usepackage{microtype}

\makeatletter
\long\def\MT@prot@l#1{%
  \MT@get@prot{#1}{left}%
  \leavevmode
  #1}
\def\MT@prot@group@@{%
    \catcode`\\=0
    \catcode`\^^M=10
    \makeatletter
    {\expandafter\MT@prot@l\expandafter
      {\expandafter\scantokens\expandafter{\the\MT@toks\noexpand}}}%
  \endgroup}
\makeatother

\usepackage[en-us]{datetime2}

\usepackage{fixmath}

\usepackage{tikz-cd}

\newtheorem{theorem}{Theorem}

\newtheorem{lemma}{Lemma}
\newtheorem{proposition}{Proposition}

\theoremstyle{definition}
\newtheorem{definition}{Definition}
\theoremstyle{remark}
\newtheorem{remark}{Remark}

\newcommand{\R}{\mathbb{R}}

\DeclareMathOperator{\spa}{span}
\DeclareMathOperator{\Ric}{Ric}

\newcommand{\h}{H^1_V(\mathcal{M})}

\newcommand{\cM}{{\mathcal M}}

\newcommand{\N}{\mathbb N}

\DeclareSymbolFont{rsfs}{U}{rsfs}{m}{n}
\DeclareSymbolFontAlphabet{\mathscr}{rsfs}

\DeclareMathOperator*{\dist}{dist}
\newcommand{\gVol}{\mathrm{Vol}}

\begin{document}

\title{Multiple solutions for {S}chr\"{o}dinger equations on {R}iemannian manifolds via \(\nabla\)-theorems\footnote{The first and third author are supported by
    GNAMPA, project ``Equazioni alle derivate parziali: problemi e
    modelli.''}}

\author{Luigi Appolloni\thanks{Dipartimento di Matematica e
    Applicazioni, Universit\`a degli Studi di Milano Bicocca. Email:
    \href{mailto:l.appolloni1@campus.unimib.it}{l.appolloni1@campus.unimib.it}}
  \and Giovanni Molica Bisci\thanks{Dipartimento di Scienze Pure e
    Applicate, Universit\`{a} di Urbino Carlo Bo. Email:
    \href{mailto:giovanni.molicabisci@uniurb.it}{giovanni.molicabisci@uniurb.it}}
  \and Simone Secchi\thanks{Dipartimento di Matematica e Applicazioni,
    Universit\`a degli Studi di Milano Bicocca. Email:
    \href{mailto:simone.secchi@unimib.it}{simone.secchi@unimib.it}}}

\date{\DTMnow}
\maketitle

\begin{abstract}
	We consider a smooth, complete and non-compact Riemannian manifold $(\mathcal{M},g)$ of dimension $d \geq 3$, and we look for positive solutions to the semilinear elliptic equation
	\begin{displaymath}
		-\Delta_g w + V w = \alpha f(w) + \lambda w \quad\hbox{in $\mathcal{M}$}.
	\end{displaymath}
	The potential $V \colon \mathcal{M} \to \mathbb{R}$ is a continuous function which is coercive in a suitable sense, while the nonlinearity $f$ has a subcritical growth in the sense of Sobolev embeddings. By means of $\nabla$-Theorems introduced by Marino and Saccon, we prove that at least three solution exists as soon as the parameter $\lambda$ is sufficiently close to an eigenvalue of the operator $-\Delta_g$.
\end{abstract}

\section{Introduction} \label{sec:1}

The study of solutions to semilinear partial differential equations of Schr\"{o}dinger type is by far one of the richest field in Nonlinear Analysis, where Variational Methods and Critical Point Theory provide a powerful setting for existence results. The occurrence of more than one solution to such equations is guaranteed, at a basic level, by some symmetry condition together with the use of topological indices like the genus or the relative category. We refer to the classical monograph \cite{MR1411681} for a survey.

Semilinear elliptic equations of Schr\"{o}dinger type are typically set in the whole Euclidean space $\mathbb{R}^d$, $d \geq 3$, which has a rather poor geometric structure. Multiplicity results may then appear as a consequence of the presence of potential functions with suitable properties. The situation is much different if $\mathbb{R}^d$ is replaced by a more general Riemannian manifold $\mathcal{M}$, since the geometry of $\mathcal{M}$ may influence the existence of one or more solutions to the equation. Analysis on Manifolds and Geometric Analysis become the necessary language to work with these problems: we refer to \cite{MR2569498,MR2954043,LiuBo,zbMATH01225960,zbMATH01447265,zbMATH03808476} and to the references therein for an introduction. For the sake of brevity, we will assume that the reader is familiar with the basic definitions of Riemannian Geometry.

We will consider a $d$-dimensional smooth complete non-compact Riemannian manifold $\left( \cM,g \right)$ with $d\geq 3$. The aim of this paper is to study the existence of solutions for problem
\begin{equation}
\tag{$P_{\lambda}$}
\begin{cases}  - \Delta_g w +V(\sigma)w=\alpha(\sigma)f(w) + \lambda w & \mbox{in} \ \cM \\
  w\geq 0 & \mbox{in} \ \cM\\
  w(\sigma) \rightarrow 0 & \mbox{as} \ d_g(\sigma_0,\sigma)\rightarrow \infty,
\end{cases}
\end{equation}
where $\alpha \in L^1(\cM) \cap L^{\infty}(\cM) \setminus \lbrace 0 \rbrace$, $\alpha \geq 0$, $f\colon [0,+\infty) \to \R$ is a continuous function, $\lambda \in \R$ is a real parameter. 
We assume that $V\colon \cM \to \R$ is a continuous function such that
\begin{itemize}
\item[$(V_1)$] $\upsilon_0:= \inf_{\sigma \in \cM} V(\sigma) >0$;
\item[$(V_2)$] there is $\sigma_0 \in \cM$ such that
\[
\lim_{d_g(\sigma_0,\sigma)\to \infty} V(\sigma)= + \infty,
\]
\end{itemize}
where $d_g\colon \cM\times\cM\rightarrow [0,+\infty)$ is the distance associated to the Riemannian metric $g$. Finally, $\Delta_g$ denotes the Laplace-Beltrami operator.

The nonlinearity $f \colon [0,+\infty) \to \mathbb{R}$ satisfies $f(0)=0$ and
\begin{itemize}
\item[$(f_1)$]
\[
\lim_{t \to 0} \frac{f(t)}{|t|}=0;
\]
\item[$(f_2)$]
there results
\[
\lim_{t \to + \infty} \frac{f(t)}{|t|^{r-1}} < \infty
\]
where $r \in (2,2d/(d-2))$;
\item[$(f_3)$] set $F(t):= \displaystyle\int_0^t f(\tau)\, d \tau$. The map
\[
t \to f(t)t-2F(t)
\]
is non decreasing if $t >0$;
\item[$(f_4)$] $0<rF(t)<f(t)t$ for all $t > 0$.
\end{itemize}
This operator is defined in local coordinates by
\[
\Delta_g h= \sum_{i,j} \frac{1}{\sqrt{\strut \det g}} \frac{\partial}{\partial x^i} \left( g^{ij}\sqrt{\det g} \frac{\partial h}{\partial x^j} \right).
\]
We point out that we have defined $\Delta_g$ with the ``analyst's sign convention'', so that $-\Delta_g$ coincides with $-\Delta$ in $\mathbb{R}^d$ with its flat metric.

\medskip

To introduce the main assumption on the manifold $(\cM,d)$, we suppose that there exists a function $H\colon \left[0,\infty \right) \to \R$ of class $C^1$ such that
\[
\int_0^\infty t H(t) \, dt < \infty
\]
and
\begin{itemize}
\item[(Ric)] for some $\bar{\sigma}_0 \in \cM$ there results
\[
\Ric_{\left(\cM,g \right)}(\sigma) \geq (1-d) H(d_g(\bar{\sigma}_0, \sigma)).
\]
\end{itemize}
Moreover, we will assume throughout the paper that
\[
\inf_{\sigma \in \cM} \gVol_g \left( B_\sigma (1)\right)>0
\]
where
\[
B_\sigma (1):= \left\lbrace \xi \in \cM \mid \dist(\xi, \sigma) < 1 \right\rbrace
\]
and
\[
\gVol_g \left( B_\sigma (1)\right) :=\int_{B_\sigma(1)} \, dv_g.
\]
Since we want to prove a multiplicity result for ($P_\lambda$), a natural approach could be based on Morse Theore, see \cite{zbMATH00217439,zbMATH00042549}. Unfortunately Morse Theory requires in general more regularity of the Euler functional associated to the variational problem, and this would require a more regular nonlinearity $f$ in ($P_\lambda$).

We propose here a different approach via~$\nabla$-Theorems, a family of variational tools which were introduced by Marino and Saccon in \cite{MR1655535} to study the multiplicity of solutions of some asymptotically non-symmetric semilinear elliptic problems with jumping nonlinearities. More precisely, we will make use of the sphere-torus linking Theorem with mixed type assumptions (see \cite[Theorem 2.10]{MR1655535}). The main condition of this theorem can be roughly summarized in these terms: the Euler functional constrained on a closed subspace must not have critical values in a certain prescribed range with ``some uniformity''. A rigorous definition is as follows.
\begin{definition} \label{def:nablacondition}
Let $\mathcal{H}$ be an Hilbert space and $\mathcal{I} \colon \mathcal{H} \to \R$ a $C^1$ functional. Let also $\mathcal{X}$ be a closed subspace of $\mathcal{H}$, $a$, $b \in \R \cup \lbrace -\infty, \infty \rbrace$; we say that $\mathcal{I}$ satisfies the condition \(\left( \nabla \right)\left( \mathcal{I}, \mathcal{X},a,b \right)\)
if there exists $\gamma >0$ such that
\begin{displaymath}
	\inf \left\lbrace \Vert P_{\mathcal{X}} \nabla \mathcal{I}(w) \Vert \mid a \leq \mathcal{I}(w)\leq b, \ \dist(w,\mathcal{X}) \leq \gamma \right\rbrace >0
\end{displaymath}
where $P_{\mathcal{X}}\colon \mathcal{H} \to \mathcal{X}$ denotes the standard orthogonal projection. In the following we will refer to it as $(\nabla)$-condition for short.
\end{definition}
Therefore we need a suitable Hilbert space in which ($P_\lambda$) can be associated to the critical points of a $C^1$ functional $\mathcal{I}$. The variational setting is described in Section \ref{sec:2}.

$\nabla$-Theorems turned out to be a powerful tool when one is interested in studying the multiplicity of solutions for nonlinear equations. In particular, in \cite{MR1708175} Pistoia proved the existence of four solutions for a superlinear elliptic problem on a bounded domain of $\R^d$. At a later time, in the same spirit of the paper of Pistoia, Mugnai proved in \cite{MR2090280} the existence of three solutions for a superlinear boundary problem with a more general nonlinearity. $\nabla$-Theorems are useful also when one deal with problems with higher order operators as showed in \cite{MR1784811} by Micheletti, Pistoia and Saccon. It is also worth mentioning \cite{MR3656482} where Molica Bisci, Mugnai and Servadei showed the existence of three solutions for an equation driven by the fractional Laplacian on a bounded domain of $\R^d$ with Dirichlet condition and a general nonlinearity. When one draws his attentions to problems settled in unbounded domains, the situation is completely different. Indeed, in order to apply the sphere-torus linking Theorem it is necessary to split the space on which is defined the functional in three linear subspaces, two of them finite dimensional, while the third infinite dimensional. When $\Omega$ is a bounded domain of $\R^d$ it is well known that the embedding $H^1(\Omega) \hookrightarrow L^2(\Omega)$ is compact. As a consequence of that, the resolvent of the Schr\"odinger operator or the Laplacian is compact and with standard arguments it is possible to prove that the spectrum of these operators is discrete and that the eigenfunctions are dense in the space under considerations. So, a common approach to select the three subspaces is to consider the whole space as direct sum of eigenspaces. Unfortunately, this strategy fails in the case of unbounded domains since the spectrum of the Schr\"odinger operator or the Laplacian is not even discrete in general. A contribution in this direction was given by Tehrani in \cite{MR1954516} where the existence of two solution for the Nonlinear Schr\"odinger equation in $\R^d$. Following the characterization of the essential spectrum of a Schr\"odinger operator present in \cite{MR1186643}, they are able to decompose the space and apply the theorem. The drawback of their approach is that they don't give sufficient condition on the potential to ensure the existence of eigenvalues subsequent to the first one. A recent result was also obtained by Mugnai in \cite{MR2402920} proving the existence of at least two solutions for an equation in which the nonlinearity is allowed to have an exponential growth in $\R^2$. In the present paper we want to extend the results quoted previously in two directions. The first one is to give sufficient condition that will enable us to completely characterize the spectrum of the operator taken into account. Secondly, the problem we want to investigate is settled in a non compact Riemannian manifold and, as far as we know, results as the one we are going to prove are not present in literature. One of the first contribute for the Nonlinear Schr\"odinger equation on Riemannian manifold was given in \cite{MR4020749}, where Faraci and Farkas established a necessary and sufficient condition for the existence of non trivial solutions with hypothesis on the manifold equal to the ones we will assume. More recently, Molica Bisci and Secchi in \cite{MR3886596} showed the existence of at least two solutions for \eqref{eq:Plambda} requiring $\lambda$ large enough under our assumptions on $f$.

\bigskip

The main result of the paper is a multiplicity result for problem \eqref{eq:Plambda} whenever $\lambda$ is sufficiently close to an eigenvalue of $-\Delta_g$.
\begin{theorem} \label{th2}
Assume $f\colon \R \to \R$ and $V\colon \cM \to \R$ are continuous functions that verify respectively $(f_1)$ -- $(f_4)$ and $(V_1)$ -- $(V_2)$. Then, for every $\lambda_k$ eigenvalue of $-\Delta_g$, there is a $\mu >0$ such that for all $\lambda \in (\lambda_k-\mu, \lambda_k)$ problem $\eqref{eq:Plambda}$ admits at least three non-trivial and non-negative weak-solutions $w_1$, $w_2$ and $w_3$. Furthermore, these solutions belong to $L^\infty(\mathcal{M})$ and for each $i \in \lbrace 1,2,3 \rbrace$ there results
\begin{equation}
	\lim_{d_g(\sigma,\sigma_0) \to +\infty} w_i (\sigma) = 0.
	\label{eq:1.1}
\end{equation}
\end{theorem}
The proof of the previous Theorem is based on a precise description of the spectral properties of the operator $-\Delta_g+V$ which governs ($P_\lambda$).  In Section \ref{sec:2} we prove in detail these properties, since they seem to be new in the setting of a non-compact manifold $\mathcal{M}$.
\begin{remark} \label{rem:1.3}
	The boundedness of our solutions and their decay at infinity \eqref{eq:1.1} follow from \cite[Theorem 3.1]{MR4020749}. This remark applies to the eigenfunctions considered in Section \ref{sec:2} as well.
\end{remark}

\medskip

To the best of our knowledge, our results are new even in the euclidean case $\mathcal{M} = \mathbb{R}^d$, $d \geq 3$. In this case, our assumptions on $V$ can be relaxed, and we can rely on some conditions introduced in \cite{zbMATH03605308} which ensure both the discreteness of the spectrum of the operator $-\Delta + V$ and the necessary compact embedding of the Sobolev space $H_V^1(\mathbb{R}^d)$. In our setting, the compactness of the embedding of $H_V^1(\cM)$ into $L^p(\cM)$ for all $p \in [2,2^*)$ follows from \cite[Lemma 2.1]{MR4020749}. As a concrete example we propose the following result.
\begin{theorem} \label{th3}
Assume $V\colon \cM \to \R$ is a function in $L^\infty_{\mathrm{loc}}(\mathbb{R}^d)$ which verifies
$V(x) \geq V_0 >0$ for almost every $x \in \mathbb{R}^d$ and
\begin{displaymath}
	\lim_{\vert x \vert \to +\infty} \int_{B_1(x)} \frac{dy}{V(y)} = 0.
\end{displaymath}
Then the same conclusions as in Theorem \ref{th2} hold for
\begin{displaymath}
\begin{cases}  
	- \Delta_g w +V(x)w=\frac{1}{\left(1+\vert x \vert^d \right)^2}   \vert w \vert^r + \lambda w & \mbox{in} \ \cM \\
  w\geq 0 & \mbox{in} \ \cM\\
  w(x) \rightarrow 0 & \mbox{as} \ \vert x \vert \rightarrow \infty,
\end{cases}
\end{displaymath}
where $2 < r < 2d / (d-2)$.
\end{theorem}

\section{Spectral properties of $-\Delta_g + V$} \label{sec:2}

In this section we are interested in studying the eigenvalues problem
\begin{equation} \label{eq:Plambda}
\begin{cases}  - \Delta_g w +V(\sigma)w=\lambda w & \mbox{in} \ \cM \\
  w\geq 0 & \mbox{in} \ \cM\\
  w(\sigma) \rightarrow 0 & \mbox{as} \ d_g(\sigma_0,\sigma)\rightarrow \infty.
\end{cases}
\end{equation}
We assume that $V\colon \cM \to \R$ is a continuous function such that
\begin{itemize}
\item[$(V_1)$] $\upsilon_0:= \inf_{\sigma \in \cM} V(\sigma) >0$;
\item[$(V_2)$] there is $\sigma_0 \in \cM$ such that
\[
\lim_{d_g(\sigma_0,\sigma)\to \infty} V(\sigma)= + \infty.
\]
\end{itemize}

We denote with $H^1_g(\cM)$ the Sobolev space obtained as the closure of $C^\infty(\mathcal{M})$ with respect to the norm
\[
\Vert w \Vert_g:= \left( \int_{\cM} \left| \nabla_g w (\sigma)\right|^2 \, dv_g +\int_{\cM}\left| w(\sigma) \right|^2\, dv_g \right)^{\frac{1}{2}},
\]
and with
\[
H_V^1(\cM):= \lbrace w \in H^1_g(\cM) \mid \Vert w \Vert^2 < \infty \rbrace
\]
where
\[
\Vert w \Vert:=  \left( \int_{\cM} \left| \nabla_g w (\sigma)\right|^2 \, dv_g +\int_{\cM} V(\sigma) \left| w(\sigma) \right|^2\, dv_g \right)^{1/2}
\]
induced by the scalar product
\[
\langle w_1,w_2\rangle:= \int_{\cM} \langle \nabla_g w_1(\sigma), \nabla_g w_2(\sigma) \rangle_g \, dv_g +\int_{\cM} V(\sigma) w_1(\sigma) w_2(\sigma) \, dv_g.
\]
We recall that under the assumptions we made on the potential and the manifold, the embedding
\(
H^1_V(\cM) \hookrightarrow L^q(\cM)
\)
is continuous for any $q \in \left[2,2^*\right]$ and compact for any $q \in \left[2, 2^* \right)$. We will say that a function $w \in H_V^1(\cM)$ is a weak solution for problem $\eqref{eq:Plambda}$ if
\begin{equation} \label{eq1}
 \int_{\cM} \langle \nabla_g w(\sigma), \nabla_g \varphi(\sigma) \rangle_g \, dv_g +\int_{\cM} V(\sigma) w(\sigma) \varphi(\sigma) \, dv_g= \lambda \int_{\cM} w(\sigma) \varphi(\sigma) \, dv_g
\end{equation}
for any $\varphi \in H^1_V(\cM) $, which can be written in a more compact way as
\[
\langle w, \varphi \rangle = \lambda \langle w,\varphi \rangle_{L^2(\cM)}.
\]
The condition $w(\sigma) \to 0$ as $d_g(\sigma,\sigma_0) \to +\infty$ follows from Remark \ref{rem:1.3}.
\begin{lemma} \label{lemma1}
Let $X_\star \subset H^1_V(\cM)$ be a weakly closed subspace. Set
\begin{equation*}
Y_\star = \lbrace w \in X_\star \mid \Vert w \Vert_{L^2(\cM)}=1 \rbrace, \quad
J(w) = \frac{1}{2} \|w\|^2.
\end{equation*}
Then there exists $w_\star \in Y_\star$ such that
\begin{equation} \label{eq2}
\min_{w \in Y_\star} J(w)=J(w_\star).
\end{equation}
Moreover, letting $\lambda_\star:=2 J(w_\star)$ we have
\begin{equation} \label{eq4}
\int_{\cM} \langle \nabla_g w_\star(\sigma), \nabla_g \varphi(\sigma) \rangle_g \, dv_g +\int_{\cM} V(\sigma) w_\star(\sigma) \varphi(\sigma) \, dv_g= \lambda_\star \int_{\cM} w_\star(\sigma) \varphi(\sigma) \, dv_g
\end{equation}
for any $\varphi \in X_\star$.
\end{lemma}
\begin{proof}
Let $(w_j)_j$ be a minimizing sequence for problem $\eqref{eq2}$, i.e.
\[
J(w_j) \to \inf_{w \in Y_\star} J(w) \geq 0 \quad \mbox{as $j \to \infty$}.
\]
Clearly, by definition of $J$, the sequence $\left( \Vert w_j \Vert \right)_j$ is bounded and up to extract a subsequence we can assume $w_j\rightharpoonup w_\star$ in $H_V^1(\cM)$. Since the embedding $H^1_V(\cM) \hookrightarrow L^2(\cM) $ is compact, we have that --- up to a subsequence ---
\[
w_j \to w_\star \quad \mbox{in} \ L^2(\cM) \quad  \mbox{as $j \to \infty$}.
\]
As a consequence $\Vert w_\star \Vert_{L^2(\cM))}=1$, and together with the weakly closedness of $X_\star$ implies $w_\star \in Y_\star$. Now, exploiting the weakly lower semicontinuity of the norm, we get
\begin{equation} \label{eq3}
J(w_\star) \leq \liminf_{j \to \infty} J(w_j)=\inf_{w \in Y_\star} J(w).
\end{equation}
Since~$w_\star \in Y_\star$, it follows from \eqref{eq3} that
\(
J(w_\star)=\min_{w \in Y_\star} J(w).
\)
In order to prove~$\eqref{eq4}$ we fix $\varepsilon \in \left(-1,1 \right)\setminus \lbrace 0 \rbrace$ and $\varphi \in X_\star$. We set  $w_\varepsilon:= (w_\star + \varepsilon \varphi)/ \Vert w_\star + \varepsilon \varphi \Vert_{L^2(\cM)}$, and we notice that $w_\varepsilon \in Y_\star$. Now, we observe that
\begin{equation}\label{eq5}
\Vert w_\star + \varepsilon \varphi \Vert_{L^2(\cM)}^2= \Vert w_\star \Vert_{L^2(\cM)}^2+2\varepsilon \int_{\cM} w_\star (\sigma) \varphi (\sigma) \, dv_g+o(\varepsilon)
\end{equation}
and
\begin{equation} \label{eq6}
\Vert w_\star+\varepsilon \varphi \Vert^2= \Vert w_\star \Vert^2+2 \varepsilon \langle w_\star, \varphi \rangle+o(\varepsilon).
\end{equation}
Putting together \eqref{eq5}, \eqref{eq6} and $\Vert w_\star \Vert_{L^2(\cM)}=1$, we get
\begin{align*}
2J(w_\star)&= \frac{\Vert w_\star \Vert^2+2 \varepsilon \langle w_\star, \varphi \rangle+o(\varepsilon)}{1+\int_{\cM} w_\star (\sigma) \varphi (\sigma) \, dv_g+o(\varepsilon)} \\
& =\left( \Vert w_\star \Vert^2+2 \varepsilon \langle w_\star, \varphi \rangle+o(\varepsilon) \right)\left( 1-\int_{\cM} w_\star (\sigma) \varphi (\sigma) \, dv_g+o(\varepsilon)  \right)\\
&= 2 J(w_\star) + 2\varepsilon \left( \langle w_\star, \varphi \rangle-2 J(w_\star) \int_\cM w_\star (\sigma) \varphi (\sigma) \, dv_g \right)+o(\varepsilon).
\end{align*}
At this point, since $w_\star$ is a minimum, we have that
\[
\lim_{\varepsilon \to 0} \frac{J(w_\varepsilon)-J(w_\star)}{\varepsilon}=0,
\]
but this is possible only if
\[
\langle w_\star, \varphi \rangle-2 J(w_\star) \int_\cM w_\star(\sigma) \varphi(\sigma) \, dv_g=0.
\]
\end{proof}
\begin{remark} \label{rem1}
Let $S:=\lbrace w_1,\ldots,w_k\rbrace$ where $k \in \N$ and and $w_i \in H^1_V (\cM) $ for any $i \in \lbrace 1,...,k\rbrace$. It is easy to verify that the subspace
\[
X_\star:=\lbrace w \in H^1_V(\cM) \, : \, \langle w, w_i \rangle=0 \quad \mbox{for } i=1,...,k \rbrace
\]
is weakly closed. Indeed, $\tilde{w} \in X_\star$ can be written as
\[
\tilde{w} = \sum_{j=1}^\infty c_j \tilde{w}_j
\]
where $\tilde{w}_j \in X_\star$ and $c_j \in \R$ for any $j \in \N$. With an easy computation we have
\[
\langle \tilde{w}, w_i \rangle=\sum_{j=1}^\infty c_j \langle \tilde{w_j}, w_i \rangle=0
\]
for any $ i \in \lbrace1,...,k \rbrace$.
\end{remark}
\begin{lemma} \label{lemma2}
If $\lambda \neq \tilde{\lambda}$ are distinct eigenvalues and $e$, $\tilde{e} \in H^1_V(\cM)$ are corresponding eigenfunctions, then
\[
\langle e, \tilde{e} \rangle=\int_\cM e(\sigma) \tilde{e}(\sigma) \, dv_g = 0.
\]
\end{lemma}
\begin{proof}
We consider $e$ as a weak solution of \eqref{eq:Plambda} and we act on \eqref{eq1} with $\tilde{e}$ obtaining
\begin{equation} \label{eq7}
\langle e, \tilde{e} \rangle=\lambda \int_\cM e(\sigma) \tilde{e}(\sigma) \, dv_g.
\end{equation}
Similarly, inverting the r\^{o}le of the eigenfunctions we get
\begin{equation} \label{eq8}
\langle \tilde{e}, e \rangle=\tilde{\lambda} \int_\cM  \tilde{e}(\sigma) e(\sigma) \, dv_g.
\end{equation}
From \eqref{eq7} and \eqref{eq8}, being the scalar product symmetric,  we can deduce
\[
(\lambda-\tilde{\lambda}) \int_\cM e(\sigma) \tilde{e} (\sigma) \, dv_g=0,
\]
thus
\begin{equation} \label{eq10}
\int_\cM e(\sigma) \tilde{e} (\sigma) \, dv_g=0.
\end{equation}
At this point, substituting \eqref{eq10} in \eqref{eq7} we also obtain
\(
\langle e, \tilde{e} \rangle=0
\)
as desired.
\end{proof}

We collect here the main spectral properties of problem \eqref{eq:Plambda}.
\begin{theorem} \label{th1}
The following statements hold true:
\begin{itemize}
\item[$(a)$] the smallest eigenvalue of problem \eqref{eq:Plambda} is positive and it can be characterized as
\begin{equation} \label{eq 11}
\lambda_1:= \min_{\substack{w \in H^1_V(\cM) \\  \Vert w \Vert_{L^2(\cM)}=1}} \Vert w \Vert^2
\end{equation}
or analogously
\[
\lambda_1:= \min_{w \in H^1_V(\cM) \setminus \lbrace 0 \rbrace} \frac{\Vert w \Vert^2}{\Vert w \Vert^2_{L^2(\cM)}};
\]
\item[$(b)$] there is a nonnegative eigenfunction $e_1 \in H^1_V(\cM)$ that is an associated eigenfunction to $\lambda_1$ where the minimum in \eqref{eq 11} is attained. Moreover, $\Vert e_1 \Vert_{L^2(\cM)}=1$ and $\lambda_1=\Vert e_1 \Vert^2$;
\item[$(c)$] the eigenvalue $\lambda_1$ is simple, i.e. if $w \in H^1_V(\cM)$ is such that
\begin{equation*}
\int_{\cM} \langle \nabla_g w_\star(\sigma), \nabla_g \varphi(\sigma) \rangle_g \, dv_g +\int_{\cM} V(\sigma) w_\star(\sigma) \varphi(\sigma) \, dv_g= \lambda_\star \int_{\cM} w_\star(\sigma) \varphi(\sigma) \, dv_g
\end{equation*}
for any $\varphi \in \h$ then there exists $\xi \in \R$ such that $w = \xi e_1$;
\item[$(d)$] the set of eigenvalues of problem \eqref{eq:Plambda} can be arranged into a sequence $(\lambda_k)_k$ such that
\[
\lambda_1<\lambda_2 \leq \lambda_3\leq...\leq \lambda_k \leq \lambda_{k+1} \leq...
\]
where $\lim_{k \to \infty} \lambda_k = +\infty$. Moreover, every eigenvalue can be characterized as
\begin{equation} \label{eq23}
\lambda_{k+1}:=\min_{\substack{w \in E_k^\bot \\ \Vert w \Vert_{L^2(\cM)}=1}} \Vert w \Vert
\end{equation}
or equivalently
\[
\lambda_{k+1}:=\min_{w \in E_k^\bot} \frac{\Vert w \Vert}{\Vert w \Vert_{L^2(\cM)}^2}
\]
where
\[
E_k:=\spa \lbrace e_1,\ldots,e_k \rbrace;
\]
\item[$(e)$] for any $k \in \N$ there is an eigenfunction $e_k \in E_{k-1}^\bot$ associated to the eigenvalue $\lambda_k$ such that the minimum in \eqref{eq23} is attained, i.e. $\Vert e_k \Vert_{L^2(\cM)}=1$ and
\begin{equation} \label{eq37}
\lambda_k = \Vert e_k \Vert^2;
\end{equation}
\item[$(f)$] the eigenfunctions $(e_k)_k$ are an orthonormal basis for $L^2(\cM)$ and an orthogonal bases for $\h$;
\item[$(g)$] each eigenvalue has finite multiplicity. Namely, if $\lambda_k$ is such that
\begin{equation} \label{eq34}
\lambda_{k-1}<\lambda_k=\ldots=\lambda_{k+h}<\lambda_{k+h+1}
\end{equation}
for some $h \in \N_0$, then $\spa \lbrace e_k,\ldots,e_{k+h} \rbrace$ is the eigenspace associated to $\lambda_k$.
\end{itemize}
\end{theorem}
\begin{proof}
 $(a)$ It suffices to apply Lemma~\ref{lemma1} with $X_\star=H^1_V(\cM)$ and to use the Lagrange multiplier rule.

$(b)$ The existence of an eigenfunction $e_1$ corresponding to the eigenvalues $\lambda_1$ is guaranteed  by part $(a)$ and as we have already seen during the proof of Lemma \ref{lemma2} assuming $\Vert e_1 \Vert_{L^2(\cM)}$ is not restrictive. In order to establish the nonnegativity, we recall (see \cite[Example 5.3]{MR2569498}) that for any $w \in H^1_V(\cM)$ we have
\begin{equation} \label{eq12}
\nabla_g w^+=
\begin{cases}
\nabla_g w & \hbox{if $w>0$} \\
0 &  \hbox{if $w \leq 0$}
\end{cases}
\quad \mbox{and} \quad
\nabla_g w^-=
\begin{cases}
0 & \hbox{if w$ \geq 0$} \\
\nabla_g w &  \hbox{if $w < 0$}.
\end{cases}
\end{equation}
In virtue of \eqref{eq 11} we get\
\[
\lambda_1 = \Vert e_1 \Vert^2= \Vert e_1^+ \Vert^2 + \Vert e_1^- \Vert^2 \geq \lambda_1 \Vert e_1^+\Vert_{L^2(\cM)}^2+\lambda_1 \Vert e_1^- \Vert_{L^2(\cM)}^2 = \lambda_1
\]
which implies
\[
\Vert e_1^+ \Vert=\lambda_1 \Vert e_1^+ \Vert_{L^2(\cM)}^2 \quad \mbox{and} \quad \Vert e_1^- \Vert^2=\lambda_1 \Vert e_1^- \Vert_{L^2(\cM)}^2
\]
that means the the infimum in \eqref{eq 11} is attained also by $e_1^+$ and $e_1^-$. Now, up to normalize the functions, we have $J(e_1)=J(e_1^+)=J(e_1^-)$ but this is not admissible unless $e_1^+=0$ or $e_1^-=0$. If we are in the second case the proof is completed, while in the first one is sufficient to take $-e_1$.

$(c)$  Suppose $e_1$, $f_1$ are two eigenfunctions associated to $\lambda_1$ with $e_1 \neq f_1$ and $\Vert e_1 \Vert_{L^2(\cM)}= \Vert f_1 \Vert_{L^2(\cM)} =1$ (we are going see that under these hypothesis the number $\xi$ stated in the theorem is $1$ while if $\Vert f_1 \Vert_{L^2(\cM)})\neq 1$ it is straightforward to verify that $\xi = 1/\Vert f_1 \Vert_{L^2(cM)}$).  Set $g:= e_1-f_1$ and observe that $g$ is an eigenfunction corresponding to $\lambda_1$. We have already proved in part $(b)$ that we are allowed to assume $g \geq 0$ which is equivalent to affirm $e_1 \geq f_1$. Taking the square on both sides we obtain
\begin{equation} \label{eq13}
e_1^2\geq f_1^2.
\end{equation}
However, we also have that
\begin{equation} \label{eq14}
\int_\cM e_1^2-f_1^2 \, dv_g = \Vert e_1 \Vert_{L^2(\cM}^2-\Vert f_1 \Vert_{L^2(\cM)}^2=1-1=0
\end{equation}
Putting together \eqref{eq13} and \eqref{eq14} we get the desired assertion.

$(d)$ We start defining
\begin{equation} \label{eq15}
\lambda_{k+1}:=\inf_{\substack{w \in E_k^\bot \\ \Vert w \Vert_{L^2(\cM)}=1}} \Vert w \Vert
\end{equation}
pointing out that
\[
E_k^\bot=\lbrace w \in \h : \langle w, e_i \rangle=0 \ \mbox{for} \ i=1,...,k \rbrace.
\]
Thanks to Lemma \ref{lemma1} the infimum in \eqref{eq15} is actually a minimum attained by a function $e_k \in E_k^\bot$ (notice that by virtue of remarf \ref{rem1} the set $E_k^\bot$ is weakly closed). Clearly $E_{k+1}^\bot \subset E_k^\bot$, and so $\lambda_{k+1}\geq \lambda_k$. To see that $\lambda_1 \neq \lambda_2$ we proceed by contradiction. Let $e_1, e_2 \in \h$ the corresponding eigenfunctions of $\lambda_1, \lambda_2$ respectively. If $\lambda_1=\lambda_2$, being $\lambda_1$ simple, there is $\xi \in \R$ such that $e_2=\xi e_1$. Now, on one hand we would have
\begin{equation} \label{eq16}
\langle e_1, e_2 \rangle=\xi \langle e_1, e_1 \rangle \neq 0.
\end{equation}
On the other hand $e_2 \in E_1^\bot$, therefore
\begin{equation} \label{eq17}
\langle e_1, e_2 \rangle=0.
\end{equation}
Noticing that \eqref{eq16} and \eqref{eq17} are not compatible proves the statement. Now, we want to show that for every $k \in \N$ the function $e_k$ where the minimum in \eqref{eq15} is attained is a weak solution of problem \eqref{eq:Plambda}. By the Lagrange multiplier's rule, it follows that
\begin{multline} \label{eq18}
\int_{\cM} \langle \nabla_g e_{k+1}(\sigma), \nabla_g \varphi(\sigma) \rangle_g \, dv_g +\int_{\cM} V(\sigma) e_{k+1}(\sigma) \varphi(\sigma) \, dv_g \\
= \lambda_{k+1} \int_{\cM} e_{k+1}(\sigma) \varphi(\sigma) \, dv_g
\end{multline}
for any $\varphi \in E_k^\bot$. Our aim is to prove the previous equality also for $ \varphi \in \h \setminus E_k^\bot$. Observing that the assertion is true for $k=1$, arguing by induction, we suppose it holds also for $2,\ldots,k$ and we show the validity for the case $k+1$. Remembering that
\[
\h=E_k \oplus E_k^\bot,
\]
every $\varphi \in \h$ can be written as $\varphi= \varphi_1+\varphi_2$ with $\varphi_1 \in E_k$ and $\varphi_2 \in E_k^2$. We can also say that there exists some constants $c_1,...,c_j \in \R$ such that
\[
\varphi_1=\sum_{j=1}^{k} c_j e_j.
\]
We test \eqref{eq18} with $\varphi_2=\varphi-\varphi_1$ and we obtain
\(
\langle e_{k+1}, \varphi_2 \rangle=\lambda_{k+1} \langle e_{k+1}, \varphi \rangle_{L^2(\cM)}.
\)
Rearranging the terms we get
\begin{equation} \label{eq19}
\langle e_{k+1}, \varphi \rangle - \lambda_{k+1} \langle e_{k+1}. \varphi \rangle_{L^2(\cM)}=\sum_{j=1}^k \left( \langle e_{k+1},e_j \rangle-\lambda_{k+1} \langle e_{k+1}, e_j \rangle_{L^2(\cM)} \right)
\end{equation}
At this point, from \eqref{eq1} with $e_i$ as weak solution and $e_{k+1}$ as test function it follows
\begin{equation} \label{eq20}
0=\langle e_i, e_{k+1} \rangle = \lambda_i \langle e_i, e_{k+1} \rangle_{L^2(\cM)}
\end{equation}
for $i=1,...,k$ where we used the inductive hypothesis and the fact that $e_{k+1} \in E_k^\bot$. Then, we substitute \eqref{eq20} into \eqref{eq19} we obtain
\[
\langle e_{k+1}, \varphi \rangle=\lambda_{k+1} \langle e_{k+1}, \varphi \rangle_{L^2(\cM)}
\]
for any $\varphi \in \h$.
Next, we prove that $\lambda_k \to \infty$ as $k \to \infty$. As an intermediate step we claim that
\begin{equation} \label{eq21}
\langle e_k, e_h \rangle=\langle e_k, e_h \rangle_{L^2(\cM)}=0
\end{equation}
for each $k \neq h$. Suppose without loss of generality that $k>h$, or $k-1\geq h$. From $e_k \in E_{k-1}^\bot \subset E_h^\bot$ follows immediately that
\[
\langle e_k, e_h \rangle=0.
\]
To conclude the proof of the claim, considering $e_k$ as a weak solution of \eqref{eq:Plambda}, we act on \eqref{eq1} with $e_h$ and we obtain
\(
\langle e_k, e_h \rangle=\lambda_k \langle e_k, e_h \rangle_{L^2(\cM)}
\)
which implies
\(
\langle e_k, e_h \rangle_{L^2(\cM)}=0
\)
since $\lambda_k \neq 0$. Now suppose by contradiction that $(\lambda_k)_k$ is bounded. From \eqref{eq1} and \eqref{eq15} follows easily that
\[
\lambda_k=\Vert e_k \Vert^2.
\]
Hence, there exists a subsequence such that $e_{k_j} \rightharpoonup e_\infty$ in $\h$ and $e_{k_j}\to e_\infty$ in $L^2(\cM)$
as $j \to \infty$. In particular, we have that $(e_{k_j})_j$ is a Cauchy sequence in $L^2(\cM)$ but the orthogonality showed in the claim asserts that
\[
\Vert e_{k_j}-e_{k_i} \Vert_{L^2(\cM)}^2= \Vert e_{k_j} \Vert_{L^2(\cM)}^2+\Vert e_{k_i} \Vert_{L^2(\cM)}^2=2.
\]
It remains to prove that the sequence $(\lambda_k)_k$ consists in all the set of possible eigenvalues for problem \eqref{eq:Plambda}. In order to do that, let us assume the existence of an eigenvalue $\lambda \notin\lbrace \lambda_k : k \in \N \rbrace$. Because of the divergence of $(\lambda_k)_k$ there is $\tilde(k) \in \N$ such that
\begin{equation} \label{eq22}
\lambda_{\tilde{k}}<\lambda<\lambda_{\tilde{k}+1}.
\end{equation}
Furthermore, in correspondence of $\lambda$ there exists an eigenfunction $e \in \h$ that up to normalize can be assumed such that $\Vert e \Vert_{L^2(\cM)}=1$. We also emphasize that exploiting Lemma \ref{lemma1} we can characterize the eigenvalue as $\lambda = 2 J(e)$. Now, we show $e \notin E_k^\bot$. Were the contrary true, we would have
\begin{equation*}
\lambda = 2 J(e) \geq 2J(e_{k+1})=\lambda_{k+1}
\end{equation*}
that is against \eqref{eq22}. If $e \notin E_k^\bot$ there exists an index $i \in \lbrace 1,...,k \rbrace$ such that $\langle e, e_i \rangle \neq 0$, but this is in contradiction with the thesis of Lemma \ref{lemma2}.

$(e)$ This is a direct consequence of part $(d)$ and Lemma \ref{lemma1}.

$(f)$ The orthogonality with respect the two scalar product was already proved in \eqref{eq21}. Now, we need to prove the following claim:
\begin{equation} \label{eq26}
\mbox{if} \ w \in \h \ \mbox{is such that} \ \langle w,e_k \rangle=0 \ \forall k \in \N \ \mbox{then} \ w=0.
\end{equation}
Were the claim false, we would be able to find $ \tilde{w} \in \h \setminus \lbrace 0 \rbrace$ such that
\begin{equation} \label{eq24}
\langle \tilde{w}, e_k \rangle=0 \quad \forall k \in \N.
\end{equation}
Without loss of generality we may assume that $\Vert \tilde{w} \Vert_{L^2(\cM)}=1$, and since $\lambda_k \to \infty$ we can find a $\tilde{k} \in \N$ such that
\[
2J(\tilde{w}) < \lambda_{\tilde{k}+1}=\min_{\substack{w \in E_{\tilde{k}}^\bot \\ \Vert w \Vert_{L^2(\cM)}=1}} \Vert w \Vert^2
\]
from which it follows that $\tilde{w} \notin E_{\tilde{k}}^\bot$. As a consequence of that, $\langle \tilde{w}, e_j \rangle \neq 0$ for some index $j \in \lbrace 1,...,\tilde{k}\rbrace$ contradicting \eqref{eq24} and showing the validity of the claim. Now, take $f \in \h$, set $\tilde{e}_i:= e_i / \Vert e_ \Vert$ and
\[
f_j:= \sum_{i=1}^j \langle f,\tilde{e}_i \rangle \tilde{e}_i
\]
Clearly $f_j \in E_j$. We denote with $v_j:=f-f_j$ and we observe that
\begin{align} \label{eq25}
0&\leq \Vert v_j \Vert= \Vert f \Vert^2 + \Vert f_j \Vert^2-2\langle f,f_j \rangle
 =\Vert f \Vert^2 + \Vert f_j \Vert^2-2 \sum_{i=1}^j \langle f,\tilde{e}_i \rangle^2 \\
&= \Vert f \Vert^2-\sum_{i=1}^j \langle f,\tilde{e}_i\rangle^2 \notag
\end{align}
where we used the orthogonality of $(e_k)_k$. From \eqref{eq25} it follows that
\[
\sum_{i=1}^j \langle f, \tilde{e}_i \rangle^2 \leq \Vert f \Vert^2
\]
from which we can deduce that $\sum_{i=1}^\infty \langle f, \tilde{e}_i \rangle^2$ is convergent. Hence, setting
\[
\tau_j=\sum_{i=1}^j \langle f, \tilde{e}_i \rangle
\]
we have that $(\tau_j)_j$ is a Cauchy sequence in $\R$. Furthermore, exploiting again the orthogonality of $(e_k)_k$ in $\h$, for any $h >j $ we have that
\[
\Vert v_h-v_j \Vert^2= \left\Vert \sum_{i=j}^h \langle f, \tilde{e}_i \rangle \tilde{e}_i  \right\Vert^2=\sum_{i=j}^h \langle f, \tilde{e}_i \rangle^2=\tau_h-\tau_j,
\]
which implies $(v_j)_j$ is a Cauchy sequence in $\h$. Recalling the completeness of $\h$ we obtain
\[
v_j \to v \quad \mbox{in} \quad \h \quad \mbox{as} \quad j \to \infty.
\]
At this point, choose $k \in \N$ and observe that of any $j \geq k$ we have
\[
\langle v_j, \tilde{e}_k \rangle= \langle f, \tilde{e}_k \rangle- \langle f_j, \tilde{e}_k \rangle= \langle f, \tilde{e}_k \rangle-\langle f, \tilde{e}_k \rangle =0.
\]
Passing to the limit as $j \to \infty$ in the above chain of equality we get
\[
\langle v, \tilde{e}_k \rangle=0 \quad \forall k \in \N.
\]
Thus, by claim  \eqref{eq26} it follows that $v=0$, and as a consequence of that
\[
f_j=f-v_j \to f \quad \mbox{in} \quad \h \quad \mbox{as} \quad j \to \infty.
\]
So, the completeness of the sequence $(e_k)_k$ in $\h$ is established. It remains to prove that $(e_k)_k$ is a basis also in the space $L^2(\cM)$. We recall that $L^2(\cM)$ is the completion of the set of smooth functions with respect to the $L^2$-norm, see \cite{LiuBo}. For any $w \in L^2(\cM)$ we select a sequence $(w_j)_j \in C_c^\infty(\cM)$ such that  \begin{equation} \label{eq27}
\Vert w_j - w \Vert_{L^2(\cM)} \leq \frac{1}{j}.
\end{equation}
Then, since  $w_j \in \h$ and $(e_k)_k$ is a basis for $\h $, we can find a $k_j \in \N$ and a $v_j \in E_{k_j}$ such that
 \begin{equation} \label{eq28}
 \Vert v_j-w_j \Vert \leq \frac{1}{j}.
 \end{equation}
  Moreover, recalling the embedding $\h \hookrightarrow L^2(\cM)$, there exists a constant $C >0$ such that
\begin{equation} \label{eq 29}
\Vert w_j-v_j\Vert_{L^2(\cM)}\leq C \Vert w_j-v_j \Vert \leq \frac{C}{j}
\end{equation}
Putting together \eqref{eq27}, \eqref{eq28} and \eqref{eq 29} we get
\begin{equation} \label{eq30}
\Vert w-v_j \Vert_{L^2(\cM)} \leq \Vert w-w_j\Vert_{L^2(\cM)}+\Vert w_j-v_j \Vert_{L^2(\cM)} \leq \frac{C+1}{j}.
\end{equation}
Letting $j \to \infty$ we see that $(e_k)_k$ is a basis also for $L^2(\cM)$.

$(g)$ The finiteness of the dimension of the eigenspaces is due to the characterization of the eigenvalues \eqref{eq23}. Now, let $h \in \N_0$ be the integer such that \eqref{eq30} holds. Set
\[
E_k^h:=E_{k+h} \setminus \operatorname{span} \lbrace e_{k-1} \rbrace
\]
and observe that trivially every function in $E_k^h$ is an eigenfunction of problem \eqref{eq:Plambda} corresponding to the eigenvalue $\lambda_k=...=\lambda_{k+h}$. On the other hand, let $\psi \neq 0$ be an eigenfunction for problem \eqref{eq:Plambda} associated to $\lambda_k$. We are going to prove that $\psi \in E_k^h$. First of all we recall that
\[
\h=E_k^h \oplus \left( E_k^h \right)^\bot,
\]
then we can write $\psi=\psi_1+\psi_2$ with $\psi_1 \in E_k^h$ and $\psi_2 \in \left( E_k^h \right)^\bot$. Of course we have that \(\langle \psi_1, \psi_2 \rangle=0.\)
From \eqref{eq1} with $\psi$ both as a solution and as a test function, using \eqref{eq15}, it follows
\begin{equation} \label{eq32}
\lambda_k \Vert \psi \Vert_{L^2(\cM)}^2=\Vert \psi \Vert^2=\Vert \psi_1 \Vert^2 + \Vert \psi_2 \Vert^2.
\end{equation}
We emphasize that, being $\psi_2 \in E_k^h$, it is an eigenfunction of problem \eqref{eq:Plambda} corresponding to the eigenvalue $\lambda_k$.
Thus, we can act on \eqref{eq1} with $\psi_1$ as a weak solution and with $\psi_2$ as a test function obtaining
\begin{equation} \label{eq33}
\lambda_k \int_\cM \psi_1(\sigma) \psi_2(\sigma)\, dv_g=\langle \psi_1, \psi_2\rangle=0,
\end{equation}
which implies
\begin{equation} \label{eq31}
\langle \psi_1,\psi_2 \rangle_{L^2(\cM)}=0,
\end{equation}
and so
\begin{equation*}
\Vert \psi \Vert_{L^2(\cM)}^2=\Vert \psi_1 \Vert_{L^2(\cM)}^2+\Vert \psi_2 \Vert_{L^2(\cM)}^2.
\end{equation*}
At this point, since $\psi_1 \in E_k^h$, we can write
\[
\psi_1 = \sum_{i=k}^{k+h} c_i e_i
\]
for some constants $c_k,...,c_{k+h} \in \R$. By parts $(e)$ and $(f)$, recalling \eqref{eq34}, we have
\begin{equation} \label{eq36}
\Vert \psi_1 \Vert^2=\sum_{i=k}^{k+h}c_i^2 \Vert e_i \Vert^2=\sum_{i=k}^{k+h} c_i^2 \lambda_i=\lambda_k \sum_{i=k}^{k+h} c_i^2=\lambda_k \Vert \psi_1 \Vert_{L^2(\cM)}^2.
\end{equation}
Now, we observe that since $\psi_2=\psi-\psi_1$ is also an eigenfunction for $\lambda_k$. Therefore, from Lemma \ref{lemma2} and \eqref{eq34} follows that
\(
0=\langle \psi_2, e_1 \rangle=...=\langle \psi_2, e_{k-1} \rangle.
\)
Thus $\psi_2 \in E_{k+h}^\bot$. In order to conclude the proof, it suffices to show $\psi_2=0$. If not, we would have that
\begin{equation} \label{eq35}
\lambda_k <\lambda_{k+h+1} = \min_{w \in E_{k+h+1}^\bot} \frac{\Vert w \Vert^2}{\Vert w \Vert_{L^2(\cM)}^2} \leq \frac{\Vert \psi_2 \Vert^2}{\Vert \psi_2 \Vert_{L^2(\cM)}^2}.
\end{equation}
Putting together \eqref{eq32}, \eqref{eq31}, \eqref{eq36} and \eqref{eq35} and  we get
\begin{align*}
\lambda_k \Vert \psi \Vert_{L^2(\cM)}^2 &=\Vert \psi_1 \Vert^2 + \Vert \psi_2 \Vert^2  \\
& > \lambda_k \Vert \psi_1 \Vert_{L^2(\cM)}^2 \lambda_k \Vert \psi_2 \Vert_{L^2(\cM)}^2 \\
&=\lambda_k \Vert \psi \Vert_{L^2(\cM)}^2,
\end{align*}
a contradiction. Hence $\psi_2=0$ and from this we deduce
\(
\psi=\psi_1 \in E_k^h
\)
completing the proof of the theorem.
\end{proof}
Next Lemma generalize the Poincaré inequality to the case in which the functions belong to eigenspaces or its orthogonal.
\begin{lemma} \label{lemma3}
Let $k \in \N$. The following inequalities hold:
\begin{itemize}
\item[$(a)$] if $w \in E_k^{\perp}$ then
\begin{equation} \label{eq38}
\Vert w \Vert^2  \geq \lambda_{k+1} \Vert w \Vert_{L^2(\cM)}^2;
\end{equation}
\item[$(b)$] if $w \in E_k$ then
\begin{equation} \label{eq39}
\Vert w \Vert^2 \leq \lambda_{k} \Vert w \Vert_{L^2(\cM)}^2.
\end{equation}
\end{itemize}
\end{lemma}
\begin{proof}
We start with the case $(a)$. Since $w \in E_k^{\perp}$ we can write
\[
w=\sum_{j=k+1}^{\infty} \alpha_j e_j
\]
for some coefficients $\alpha_j \in \R$. Thus, we compute
\[
\Vert w \Vert^2 = \langle w,w \rangle =\sum_{j=k+1}^{\infty} \alpha_j^2 \lambda_j \geq \lambda_{k+1} \Vert w \Vert^2_{L^2(\cM)}
\]
where we used Theorem \ref{th1} $(f)$, \eqref{eq37} and the Bessel-Parseval's identity (see for instance \cite[Theorem 5.9]{MR2759829}). On the other hand, when $w \in E_k$ we have
\[
w=\sum_{j=1}^k \alpha_j e_j.
\]
As a consequence, similarly as we did above we get
\[
\Vert w \Vert^2 = \sum_{j=1}^k \alpha_j^2 \lambda_j \leq \lambda_k \Vert w \Vert_{L^2(\cM)}^2.
\]
\end{proof}

\section{A setting for ($P_\lambda$)}

Let us consider
\begin{equation*}
\begin{cases}  - \Delta_g w +V(\sigma)w=\alpha(\sigma)f(w) + \lambda w & \mbox{in} \ \cM \\
  w\geq 0 & \mbox{in} \ \cM\\
  w(\sigma) \rightarrow 0 & \mbox{as} \ d_g(\sigma_0,\sigma)\rightarrow \infty,
\end{cases}
\end{equation*}
where $\alpha \in L^1(\cM) \cap L^{\infty}(\cM) \setminus \lbrace 0 \rbrace$ is a nonnegative function. $f\colon  \R \to \R$ is a continuous  function such that $f(0)=0$ that satisfies
\begin{itemize}
\item[$(f_1)$]
\[
\lim_{t \to 0} \frac{f(t)}{|t|}=0;
\]
\item[$(f_2)$]
there results
\[
\lim_{t \to + \infty} \frac{f(t)}{|t|^{r-1}} < \infty
\]
where $r \in (2,2d/(d-2))$;
\item[$(f_3)$] set $F(t):= \displaystyle\int_0^t f(\tau)\, d \tau$. The map
\[
t \to f(t)t-2F(t)
\]
is non decreasing if $t >0$;
\item[$(f_4)$] $0<rF(t)<f(t)t$ for all $t >0$.
\end{itemize}
\begin{remark}
The behavior of $f$ on the half-line $(-\infty,0)$ is irrelevant here. Indeed, since we are interested only in looking for positive solutions, we can set
\[
f^+(t):=
\begin{cases}
f(t) & t>0 \\
0 & t \leq 0.
\end{cases}
\]
For the sake simplicity, since now to the end of the paper, we will write $f$ instead of $f^+$.
\end{remark}
\begin{remark}
If $f \in C^1(\mathbb{R})$, assumption $(f_3)$ can be formulated as
\begin{itemize}
\item[$(f_3')$] The map $t \mapsto F(t)/t$ is non-decreasing for $t>0$.
\end{itemize}
This shows that $(f_3)$ is a reasonable assumption that extends $(f_3')$ in the less regular case we are considering.
\end{remark}
In order to find solutions for problem \eqref{eq:Plambda} we introduce the energy functional associated to the problem. Namely, let $J_{\lambda}(w):\h \to \R$ where
\[
J_{\lambda}(w) = \frac{1}{2} \Vert w \Vert - \frac{\lambda}{2} \Vert w \Vert^2_{L^2(\cM)}-\int_{\cM} \alpha(\sigma) F(w(\sigma))\, dv_g.
\]
In virtue of the embedding results presented in the previous sections this functional is well defined and it is standard to prove that is $C^1$. Moreover, as it is well known, critical points of $J_\lambda$ correspond to weak solutions of problem \eqref{eq:Plambda}, i.e.
\[
\langle w, \varphi \rangle = \lambda \langle w,\varphi\rangle_{L^2(\cM)} +\int_\cM \alpha(\sigma) f(w(\sigma)) \varphi(\sigma)\, dv_g
\]
for any $\varphi \in \h$. More in general, one can shows that the derivative of the functional $J_\lambda$ along a function $v \in \h$ is
\begin{equation} \label{eq49}
J'_\lambda(w) \left[ w \right]= \langle w,v \rangle - \lambda \langle w,v \rangle_{L^2(\cM)}-\int_{\cM} \alpha(\sigma) f(w(\sigma)) v(\sigma)\, dv_g.
\end{equation}
Now, take $s \in \left[2,2^*\right)$ and consider its conjugate exponent $s'$ such that $1/s+1/s'=1$. We select a function $ h \in L^{s'}(\cM)$ and we focus on the equation
\begin{equation} \label{eq48}
-\Delta_g w = h(\sigma) \quad \in \cM.
\end{equation}
By applying the classical Riesz or Lax-Milgram Theorem, one can easily show that the problem above has a unique weak solution. In virtue of that, we are able to define
\begin{eqnarray*}
\Delta_g^{-1}: L^{s'}(\cM) &\rightarrow& \h \\
h & \mapsto & w=\Delta_g^{-1} h
\end{eqnarray*}
where $\Delta_g^{-1} h$ is the only weak solution of \eqref{eq48}, that means
\begin{equation}\label{eq50}
\langle \Delta_g^{-1} h, \varphi \rangle=\langle h, \varphi \rangle_{L^2(\cM)}.
\end{equation}
We emphasize that the operator $\Delta_g^{-1}$ is compact. Indeed, it is possible to write it by the composition of two maps
\[
\begin{tikzcd}
L^{s'}(\cM)  \arrow[r, "j"] & \left(\h\right)^* \arrow[r, "\Delta_g^{-1}"] & \h
\end{tikzcd}
\]
where the first is compact recalling that $\h \hookrightarrow L^s(\cM)$ is compact and applying \cite[Theorem 6.4]{MR2759829}. Since $\h$ is a Hilbert space, there is a unique element called the gradient of $J_\lambda$ and denoted $\nabla J_\lambda$ such that
\begin{equation} \label{eq51}
\langle \nabla J_\lambda (w),v \rangle=J_\lambda'(u) \left[ v \right].
\end{equation}
It is also possible to verify that the gradient of $J_\lambda$ can be written as
\begin{equation} \label{eq52}
\nabla J_\lambda (w)=w-\Delta_g^{-1}\left( \lambda w +\alpha f(w) \right).
\end{equation}

We begin our analysis proving a technical lemma that will provide some useful estimates we will use throughout the paper.
\begin{lemma} \label{lemma4}
If $f \colon \left[0,\infty \right) \to \R$ is a function that satisfies $(f_1)$ -- $(f_4)$, then we have the following estimates:
\begin{itemize}
\item[$(i)$] for any $\varepsilon >0$ there exists a constant $A_1^\varepsilon >0$ such that
\begin{equation} \label{eq40}
f(t) \leq 2\varepsilon t + rA_1^\varepsilon t^{r-1}
\end{equation}
and
\begin{equation} \label{eq41}
F(t) \leq \varepsilon t^2 + A_1^\varepsilon t^r
\end{equation}
for any $ t \geq 0$;
\item[$(ii)$] for any $\varepsilon >0$ there exist $A_2, A_2^{\tilde{\varepsilon}} >0$ such that
\begin{equation} \label{eq42}
f(t) \leq A_2+A_2^{\varepsilon} t^{r-1}
\end{equation}
for any $t \geq 0$;
\item[$(iii)$] there exists $A_3, A_4 >0$ such that
\begin{equation} \label{eq43}
F(t) \geq A_3 t^r-A_4
\end{equation}
for any $t \geq 0$
\end{itemize}
\end{lemma}
\begin{proof}
By  $(f_1) $ and $(f_2)$, for each $\varepsilon >0$ we can find $0 <\gamma_\varepsilon< \theta_\varepsilon$ such that
\begin{equation} \label{eq44}
f(t) \leq 2\varepsilon t
\end{equation}
for any $ t \in \left[0, \gamma_{\varepsilon} \right)$ and
\begin{equation} \label{eq45}
f(t) \leq r A_1 t^{r-1}
\end{equation}
for any $t \geq \theta_\varepsilon$ and for some $A_1>0$. Furthermore, exploiting the continuity of $f$, we can find a constant $K_\varepsilon>0$ such that
\begin{equation} \label{eq46}
\frac{f(t)}{t^{r-1}} \leq r K_\varepsilon
\end{equation}
for any $t \in \left[\gamma_\varepsilon, \theta_\varepsilon \right)$. Putting together \eqref{eq44}, \eqref{eq45}, \eqref{eq46} and relabelling the constants we obtain \eqref{eq40}. Integrating \eqref{eq40} we get \eqref{eq41}. It is also possible to find $A_2 >0$ such that
\begin{equation} \label{eq47}
f(t) \leq A_2
\end{equation}
for any $t \in \left[0, \theta_\varepsilon \right)$. Coupling \eqref{eq45} and \eqref{eq47} we have \eqref{eq42}. In order to prove the last part of the lemma we fix $R >0$. We reorder the terms in $(f_4)$ to have
\[
\frac{r}{t} \leq \frac{f(t)}{F(t)}.
\]
Integrating between $R$ and $t$ and using the properties of the logarithm we obtain
\[
F(t) \geq A_3 t^r
\]
for some $A_3>0$ and any $t \geq R$. Since there is $A_4>0$ such that
\[
F(t) \geq -A_4
\]
for any $0 \leq t < R$ the proof of the lemma ends gathering the last two equations.
\end{proof}
We end this section by proving that the functional $J_\lambda$ satisfies a good compactness condition in Critical Point Theory.

\begin{definition}
We say that a sequence $(w_j)_j \subset \h$ is a Palais-Smale sequence at level $c \in \R$, $(PS)_c$ sequence for short, if $J_\lambda (w_j) \to c$ in $\R$ and $J_\lambda'(w_j) \to 0$ in $\left( \h \right)^*$ as $j \to \infty$. Furthermore, the functional $J_\lambda$ is said to satisfy the $(PS)_c$ condition if every $(PS)_c$ sequence for $J_\lambda$ admits a strongly convergent subsequence in $\h$.
\end{definition}
\begin{proposition} \label{prop1}
Let $f$ a map satisfying $(f_1)$--$(f_4)$ and $\lambda >0$ a real parameter. Then, $(PS)_c$ condition holds for every $c \in \R$ for functional  $J_\lambda$.
\end{proposition}
\begin{proof}
Let $(w_j)_j \subset \h$ a $(PS)_c$ sequence for functional $J_\lambda$, i.e.
\begin{equation} \label{eq53}
J_\lambda(w_j) \to c \quad \mbox{in} \ \R
\end{equation}
and
\begin{equation} \label{eq54}
J_\lambda'(w_j) \to 0 \quad \mbox{in} \ \h
\end{equation}
as $ j \to \infty$. We observe that
\begin{equation} \label{eq55}
2J_\lambda (w_j)-J_\lambda'(w_j) \left[ w_j \right] = \int_\cM \alpha(\sigma) \left[ f(w_j(\sigma) w_j (\sigma)-2F(w_j(\sigma)) \right] \, dv_g.
\end{equation}
Hence, collecting \eqref{eq53}, \eqref{eq54}, \eqref{eq55}, we can find a $C_1>0$ such that
\begin{equation} \label{eq56}
\int_\cM \alpha(\sigma) \left[ f(w_j(\sigma) w_j (\sigma)-2F(w_j(\sigma)) \right]\, dv_g. \leq C_1
\end{equation}
for all $j \in \N$.

\textit{Claim}: the sequence $(w_j)_j$ is bounded in $\h$.

If not, we might assume that $\Vert w_j \Vert \to \infty$ as $j \to \infty$. Let us set
\(
u_j:=\frac{w_j}{\Vert w_j \Vert}.
\)
Since $(u_j)_j$ is bounded in $\h$ we can assume
\[
u_j \rightharpoonup u \quad \mbox{in} \ \h,
\]
and from \cite[Lemma 2.1]{MR4020749}
we get
\[
u_j \to u \quad \mbox{in $L^s(\cM)$}
\]
for all $s \in \left[2, 2^*\right)$. We distinguish the cases $u \neq 0$ and $u =0$. In the first case we define
\[
\cM_0:=\lbrace \sigma \in \cM \mid u(\sigma) =0 \rbrace
\]
and clearly we have
\[
\lim_{j \to \infty} w_j(\sigma)= \infty \quad \mbox{for a.e. $\sigma \in \cM_0^{c}$}.
\]
From \ref{lemma4} $(iii)$ it is straightforward to verify
\[
\lim_{t \to \infty} \frac{F(t)}{t^2}=\infty
\]
thus, applying the Fatou's Lemma, we get
\[
\lim_{j \to \infty} \int_{\cM_0^c} \alpha(\sigma) \frac{F(w_j(\sigma))}{\Vert w_j \Vert^2}\, dv_g=\infty.
\]
But
\[
\int_{\cM} \alpha(\sigma) \frac{F(w_j(\sigma))}{\Vert w_j \Vert^2}\, dv_g=\int_{\cM_0} \alpha(\sigma) \frac{F(w_j(\sigma))}{\Vert w_j \Vert^2}\, dv_g+\int_{\cM_0^c} \alpha(\sigma) \frac{F(w_j(\sigma))}{\Vert w_j \Vert^2}\, dv_g
\]
and so
\begin{equation} \label{eq57}
\lim_{j \to \infty}\int_{\cM} \alpha(\sigma) \frac{F(w_j(\sigma))}{\Vert w_j \Vert^2}\, dv_g=\infty.
\end{equation}
Conversely, we note that
\[
\frac{J_\lambda (w_j)}{\Vert w_j \Vert^2}= \frac{1}{2}-\frac{\lambda \Vert w_j \Vert_{L^2(\cM)}^2}{2\Vert w_j \Vert^2}-\int_\cM \alpha(\sigma)\frac{F(w_j(\sigma))}{\Vert w_j \Vert^2} \, dv_g
\]
and recalling $ \h \hookrightarrow L^2(\cM)$ we obtain that $ \Vert w_j \Vert_{L^2(\cM)}/\Vert w_j \Vert$ is bounded. These two facts and \eqref{eq53} imply the existence of $C_2>0$ such that
\begin{equation} \label{eq58}
\int_\cM \alpha(\sigma)\frac{F(w_j(\sigma))}{\Vert w_j \Vert^2} \, dv_g \leq C_2
\end{equation}
Comparing \eqref{eq57} and \eqref{eq58} we get a contradiction. If $u=0$, we fix $\eta>0$ and we define
\(
v_j:=\sqrt{2\eta} \, u_j.
\)
Clearly
\begin{equation} \label{eq59}
v_j \to 0  \quad \mbox{in} \ L^s(\cM)
 \end{equation}
as $j \to \infty$,  and from Lemma~\ref{lemma4} $(i)$ and the general Lebesgue dominated convergence Theorem  \cite[Section 4.4, Theorem 19]{MR1013117} it follows that
\begin{equation} \label{eq60}
\lim_{j \to \infty} \int_\cM \alpha(\sigma) F(v_j) \, dv_g=0.
\end{equation}
Since $\Vert w_j \Vert \to \infty$ as $j \to \infty$ it is possible to find $j_1 \in \N$ such that
\[
0<\sqrt{2\eta} \frac{1}{\Vert w_j \Vert} \leq 1
\]
for every $j \geq j_1$. Now, let
\[
J_\lambda (\xi_j w_j)=\max_{0\leq \xi \leq 1} J_{\lambda}(\xi w_j)
\]
for some $\xi_j \in \left[0,1 \right]$. First of all, we notice
\[
J_\lambda (\xi_j w_j) \geq J_{\lambda}(v_j)= \eta \left(1-\Vert v_j \Vert_{L^2(\cM)}^2\right)-\int_\cM \alpha(\sigma) F(v_j(\sigma))\, dv_g
\]
for every $j \geq j_1$. Letting $j \to \infty$, recalling \eqref{eq59} and \eqref{eq60}, we obtain
\[
\liminf_{j \to \infty} J_{\lambda} (\xi_j w_j) \geq \eta.
\]
Being $\eta$ is arbitrary, we deduce
\begin{equation} \label{eq61}
\lim_{j \to \infty} J_\lambda(\xi_j w_j)=\infty.
\end{equation}
Secondly, from the facts that $J_\lambda (0)=0$ and $J_\lambda (w_j)$ is bounded we can deduce that actually $\xi_j  \in (0,1)$. Then
\[
\frac{d}{d\xi}J_\lambda(\xi w_j)\Bigr|_{\substack{\xi=\xi_j}}=0,
\]
which can be rewritten as
\begin{equation} \label{eq62}
\Vert \xi_j w_j \Vert^2-\lambda \Vert \xi_j w_j \Vert^2_{L^2(\cM)}=\int_\cM \alpha(\sigma) f(\xi_j w_j(\sigma))\xi_j w_j(\sigma) \, dv_g.
\end{equation}
Thus, from \eqref{eq56}, \eqref{eq62} and $(f_3)$, it follows that
\begin{align} \label{eq63}
2J_\lambda(\xi_jw_j)&= \int_\cM \alpha(\sigma) \left[ f(\xi_j w_j(\sigma)) \xi_j-w_j(\sigma)-F(w_j(\sigma))\right]\, dv_g \notag \\
& \leq \int_\cM \alpha(\sigma) \left[ f( w_j(\sigma)) -w_j(\sigma)-F(w_j(\sigma))\right]\, dv_g \leq C_1
\end{align}
Of course \eqref{eq61} and \eqref{eq63} are not compatible, therefore the sequence $(w_j)_j$ must be bounded. We can now use  \eqref{eq52} and  \cite[Proposition 2.2]{MR1411681} to conclude the proof.
\end{proof}
\section{Geometry of the $\nabla$-Theorem}
As mentioned at the beginning of the paper, our aim is to prove an existence result through the so-called $\nabla$-Theorem. In order to apply this tool, it is necessary to split the space in three closed subspace, two of finite dimension and one of infinite dimension. Furthermore, the functional is required to have a precise geometrical structure. In this section we are going to show that the functional $J_\lambda$ associated to problem \eqref{eq:Plambda} possesses these properties under the assumption we made on the nonlinearity $f$ and the potential $V$. Before doing that, for the sake of simplicity we fix some notation. Henceforth, $k$ and $h$ will be positive integers such that
\[
\lambda_{k-1}<\lambda_k=\ldots=\lambda_{k+h} < \lambda_{k+h+1}.
\]
We define
\[
X_1:=E_{k-1}, \quad X_2:=\spa \lbrace e_k, \ldots e_{k+h} \rbrace, \quad X_3:=E_{k+h}^\bot.
\]
We point out that the existence of such integers $h$ and $k$ is guaranteed by Theorem \ref{th1}.

Next Proposition will show the functional $J_\lambda$ verifies the desired geometrical property we need to apply the $\nabla$-Theorem.
\begin{proposition} \label{prop2}
If assumptions $(f_1)$ -- $(f_4)$ hold and $\lambda \in (\lambda_{k-1},\lambda_k$ then there are $\rho,R \in \R$, with $R>\rho>0$ such that
\[
\sup_{\lbrace w \in X_1 \mid \Vert w \Vert \leq R\rbrace \cup \lbrace w \in X_1 \oplus X_2 \mid \Vert w \Vert=R \rbrace} J_{\lambda} < \inf_{ \lbrace w \in X_2 \oplus X_3 \mid \Vert w \Vert =\rho\rbrace}  J_{\lambda}
\]
\end{proposition}
\begin{proof}
We start showing
\[
\inf_{ \lbrace w \in X_2 \oplus X_3 \mid \Vert w \Vert =\rho\rbrace}  J_{\lambda} >0
\]
choosing $\rho$ adequately and observing that
\(
X_2 \oplus X_3=E_{k-1}^\bot.
\)
Applying twice the H\"older inequality, we get
\begin{equation} \label{eq64}
\int_\cM \alpha(\sigma) |w(\sigma)|^2 \, dv_g \leq \Vert \alpha \Vert_{L^{\frac{2^*}{2^*-2}}(\cM)} \Vert w \Vert_{L^{2^*}(\cM)}^2
\end{equation}
and
\begin{equation} \label{eq65}
\int_\cM \alpha(\sigma) |w(\sigma)|^r \, dv_g \leq \Vert \alpha \Vert_{L^{\frac{2^*}{2^*-r}}(\cM)} \Vert w \Vert_{L^{2^*}(\cM)}^r.
\end{equation}
From Lemma \ref{lemma4} $(i)$, \eqref{eq64} and \eqref{eq65} we obtain
\begin{align*}
J_\lambda (w) & \geq \frac{1}{2}\Vert w \Vert^2-\frac{\lambda}{2}\Vert w \Vert^2_{L^2(\cM)}-\varepsilon \int_\cM \alpha(\sigma) |w(\sigma)|^2\, dv_g-A_1^\varepsilon \int_\cM \alpha(\sigma) |w(\sigma)|^r\, dv_g \\
& \geq \frac{1}{2}\Vert w \Vert^2-\frac{\lambda}{2}\Vert w \Vert^2_{L^2(\cM)}-\varepsilon \Vert \alpha \Vert_{L^{\frac{2^*}{2^*-2}}(\cM)} \Vert w \Vert_{L^{2^*}(\cM)}^2 - A_1^\varepsilon \Vert \alpha \Vert_{L^{\frac{2^*}{2^*-r}}(\cM)} \Vert w \Vert_{L^{2^*}(\cM)}^r.
\end{align*}
Now, recalling $\h \hookrightarrow L^s(\cM)$ for every $s \in \left[2,2^*\right)$, it is possible to find $C>0$ such that
\[
J_\lambda(w) \geq \frac{1}{2}\Vert w \Vert^2-\frac{\lambda}{2}\Vert w \Vert^2_{L^2(\cM)}-\varepsilon C \Vert \alpha \Vert_{L^{\frac{2^*}{2^*-2}}(\cM)} \Vert w \Vert^2 - A_1^\varepsilon C \Vert \alpha \Vert_{L^{\frac{2^*}{2^*-r}}(\cM)} \Vert w \Vert^r.
\]
Finally, Lemma \ref{lemma3} yields
\[
J_\lambda (w)\geq \left[\frac{1}{2}\left( 1-\frac{\lambda}{\lambda_k} \right)-\varepsilon C \Vert \alpha \Vert_{L^{\frac{2^*}{2^*-2}}(\cM)} \right] \Vert w \Vert^2-A_1^\varepsilon C \Vert \alpha \Vert_{L^{\frac{2^*}{2^*-r}}(\cM)} \Vert w \Vert^r.
\]
At this point, choosing $\varepsilon>0$ such that
\[
\frac{1}{2}\left( 1-\frac{\lambda}{\lambda_k} \right)-\varepsilon C \Vert \alpha \Vert_{L^{\frac{2^*}{2^*-2}}(\cM)} >0
\]
and $\rho$ sufficiently small, the desired assertion is proved. On the other hand, it is possible to prove
\[
\sup_{\lbrace w \in X_1 \mid \Vert w \Vert \leq R\rbrace \cup \lbrace w \in X_1 \oplus X_2 \mid \ \Vert w \Vert=R \rbrace} J_{\lambda} \leq 0.
\]
Indeed, in the case $w \in X_1$, from Lemma \ref{lemma3} and $(f_4)$, recalling $\alpha \geq 0$ for a.e. $ \sigma \in \cM$, it follows that
\[
J_\lambda (w) \leq \frac{\lambda_{k-1}-\lambda}{2} \Vert w \Vert^2_{L^2(\cM)} \leq 0.
\]
Instead, when $w \in X_1\oplus X_2$ it suffices to use Lemma \ref{lemma4} $(iii)$ to obtain
\[
J_\lambda(w) \leq \frac{1}{2} \Vert w \Vert^2-A_3\int_\cM\alpha(\sigma) |w(\sigma)|^r\, dv_g+A_4\Vert \alpha \Vert_{L^1(\cM)}.
\]
Since $X_1 \oplus X_2$ has finite dimension all  norms are equivalent, then choosing $R>0$ big enough it is straightforward to see that \(r>2\) implies
\(
J_\lambda (w) \leq 0.
\)
\end{proof}
\section{Validity of the $(\nabla)$-condition}
This section is devoted to show the validity of the $(\nabla)$-condition  introduced in Definition \ref{def:nablacondition}. Before proving the main result of this section, we need two preliminary lemmas.
\begin{proposition} \label{prop3}
Assume Hypothesis $(f_1)-(f_4)$ hold, Then for every $\varrho >0$ there exists $\delta_\varrho>0$ such that for each $\lambda \in \left[ \lambda_{k-1}+\varrho,\lambda_{k+h+1}-\varrho \right]$ the only critical point of $J_\lambda$ constrained on $X_1 \oplus X_3$ with $J_\lambda \in \left[ -\delta_\varrho,\delta_\varrho\right]$ is the trivial one.
\end{proposition}
\begin{proof}
By contradiction we suppose the statement false. So, we are allowed to assume the existence of $\tilde{\varrho}>0$, two sequences $\mu_j \subset \left[ \lambda_{k-1}+\tilde{\varrho}, \lambda_{k+h+1}-\tilde{\varrho} \right]$ and $(w_j)_j \subset X_1 \oplus X_3$ of critical points, i.e.
\begin{equation} \label{eq66}
\langle \nabla J_{\mu_j} (w_j), \varphi \rangle=0 \quad \mbox{for any} \ \varphi \in X_1\oplus X_3
\end{equation}
such that
\begin{equation} \label{eq67}
J_{\mu_j}(w_j) \to 0 \quad \mbox{as} \ j \to \infty.
\end{equation}
Since $(w_j)_j \subset X_1 \oplus X_3$, we can take $\varphi =w_j$ as test function in \eqref{eq66}. As a result of that, we have\begin{equation} \label{eq68}
0=\Vert w_j \Vert^2-\mu_j \Vert w_j \Vert_{L^2(\cM)}^2-\int_\cM \alpha(\sigma) f(w_j(\sigma))w_j(\sigma)\,dv_g
\end{equation}
Then, we notice that \eqref{eq68} can be rewritten as
\[
0=2J_{\mu_j} (w_j)+2 \int_\cM \alpha(\sigma)F(w_j(\sigma))\, dv_g-\int_\cM \alpha(\sigma) f(w_j(\sigma))w_j(\sigma)\,dv_g.
\]
Exploiting $(f_4)$ in \eqref{eq68} we obtain
\begin{equation} \label{eq69}
0 \leq 2J_{\mu_j}(w_j)+ (2-r)\int_\cM \alpha(\sigma) F(w_j(\sigma))\, dv_g.
\end{equation}
Reordering the terms in \eqref{eq69} we get
\begin{equation} \label{eq70}
0 \leq (r-2)\int_\cM \alpha(\sigma) F(w_j(\sigma))\, dv_g \leq 2J_{\mu_j}(w_j).
\end{equation}
Putting together \eqref{eq67} and \eqref{eq70} we obtain
\begin{equation} \label{eq71}
\lim_{j \to \infty} \int_\cM \alpha(\sigma) F(w_j(\sigma))\, dv_g=0.
\end{equation}
Now, recalling $w_j \in X_1 \oplus X_3$ for all $j \in \N$, we are able to find $w_{1,j} \in X_1$ and $w_{3,j}\in X_3$ such that
\(
w_j=w_{1,j}+w_{3,j}
\).
At this point, on one hand we test \eqref{eq66} with $\varphi=w_{1,j}-w_{3,j}$ and exploiting the properties of orthogonality of $w_{1,j}$ and $w_{3,j}$ we have
\begin{align} \label{eq72}
0&=\langle \nabla J_{\mu_j}(w_j),w_{1,j}-w_{3,j}\rangle \notag \\
&= \Vert w_{1,j} \Vert^2-\Vert w_{3,j}\Vert^2-\mu_j\Vert w_{1,j}\Vert_{L^2(\cM)}^2+\mu_j \Vert w_{3,j}\Vert_{L^2(\cM)}^2 \\
&\qquad {} -\int_\cM \alpha(\sigma) f(w_j(\sigma))(w_{1,j}(\sigma)-w_{3,j}(\sigma))\,dv_g. \notag
\end{align}
Rearranging \eqref{eq72} and applying Lemma \ref{lemma3} we get
\begin{align} \label{eq73}
\int_\cM \alpha(\sigma) f(w_j(\sigma))(w_{1,j}(\sigma)-w_{3,j}(\sigma))\,dv_g &= \Vert w_{1,j} \Vert^2-\Vert w_{3,j}\Vert^2-\mu_j\Vert w_{1,j}\Vert_{L^2(\cM)}^2 \notag \\
&+\mu_j \Vert w_{3,j}\Vert_{L^2(\cM)}^2 \notag \\
&\leq \Vert w_{1,j} \Vert^2-\Vert w_{3,j}\Vert^2-\frac{\mu_j}{\lambda_{k-1}}\Vert w_{1,j}\Vert^2\notag \\
&+\frac{\mu_j}{\lambda_{k+h+1}} \Vert w_{3,j}\Vert^2  \\
&=\frac{\lambda_{k-1}-\mu_j}{\lambda_{k-1}}\Vert w_{1,j}\Vert^2+\frac{\mu_j-\lambda_{k+h+1}}{\lambda_{\lambda_{k+h+1}}} \Vert w_{3,j} \Vert^3 \notag \\
&<-\frac{\tilde{\varrho}}{\lambda_{k-1}}\Vert w_{1,j}\Vert^2-\frac{\tilde{\varrho}}{\lambda_{k+h+1}}\Vert w_{3,j}\Vert^2 \notag \\
&<-\frac{2\tilde{\varrho}}{\lambda_{k+h+1}}\Vert w_j \Vert^2.
\end{align}
On the other hand, thanks to H\"older and the continuous embedding $\h \hookrightarrow L^r(\cM)$, we have
\begin{align} \label{eq74}
\left| \int_\cM \alpha(\sigma) f(w_j(\sigma))(w_{1,j}(\sigma)-w_{3,j}(\sigma))\,dv_g \right| &\leq \Vert \alpha f(w_j) \Vert_{L^{r'}(\cM)} \Vert w_{1,j}-w_{3,j}\Vert_{L^r(\cM)} \notag \\
&\leq C \Vert \alpha f(w_j)\Vert_{L^{r'}(\cM)} \Vert w_j \Vert
\end{align}
for some $C>0$, where we used
\[
\langle w_{1,j}-w_{3,j},w_{1,j}-w_{3,j}\rangle =\Vert w_{1,j}\Vert^2-\Vert w_{3,j} \Vert^2=\Vert w_j \Vert^2.
\]
Coupling \eqref{eq73} and \eqref{eq74} we have
\[
-C \Vert \alpha f(w_j) \Vert_{L^{r'}(\cM)} \Vert w_j \Vert \leq - \frac{2 \tilde{\varrho}}{\lambda_{k+h+1}}\Vert w_j \Vert^2
\]
from which it follows that
\begin{equation} \label{eq75}
\frac{2 \tilde{\varrho}}{\lambda_{k+h+1}} \Vert w_j \Vert \leq C \Vert \alpha f(w_j) \Vert_{L^{r'}(\cM)}
\end{equation}
Then, we use Lemma \ref{lemma4} $(ii)$ and we obtain
\begin{equation} \label{eq76}
\int_\cM \left| \alpha (\sigma)f(w_j(\sigma)) \right| ^{r'}\, dv_g \leq \int _\cM \left[ \alpha(\sigma) \left( A_2+A_2^{\varepsilon} |w_j|^{r-1} \right)\right]^{\frac{r}{r-1}}.
\end{equation}
Recalling that for any $a,b \geq 0$ we have
\[
(a+b)^r\leq 2^r(a^r+b^r),
\]
from \eqref{eq76} it follows
\begin{equation} \label{eq77}
\int_\cM \left| \alpha (\sigma)f(w_j(\sigma)) \right| ^{r'}\, dv_g \leq \left(2 A_2 \Vert \alpha \Vert_{L^{r'}(\cM)} \right)^{r'}+\left(2A_2^\varepsilon \right)^{r'}\int_\cM \left(\alpha(\sigma) \right)^{r'}|w_j|^r\, dv_g.
\end{equation}
Finally, we exploit Lemma \ref{lemma4} in \eqref{eq77} and we obtain
\begin{align} \label{eq78}
\int_\cM \left| \alpha (\sigma)f(w_j(\sigma)) \right| ^{r'}\, dv_g & \leq \left(2 A_2 \Vert \alpha \Vert_{L^{r'}(\cM)} \right)^{r'} + \frac{A_4}{A_3} \left(2 A_2^\varepsilon \right)^{r'} \Vert \alpha\Vert_{L^{r'}(\cM)}^{r'} \notag \\
& + \left(2 A_2^\varepsilon \right)^{r'} \frac{A_4}{A_3} \Vert \alpha\Vert_{L^{\infty}(\cM)}^{r'-1} \int_{\cM} \alpha(\sigma)F(w_j(\sigma))\, dv_g.
\end{align}
From \eqref{eq75}, keeping into account \eqref{eq71} and \eqref{eq78}, we can deduce $(w_j)_j$ is bounded in $\h$. Hence, up to a subsequence
\[
w_j \rightharpoonup w_{\infty} \quad \mbox{in} \ \h.
\]
Furthermore, recalling that $\h \hookrightarrow L^r(\cM)$ is compact, we have
\[
w_j \to w_{\infty} \quad \mbox{in} \ L^r(\cM)
\]
\[
w_j(\sigma) \to w_{\infty}(\sigma) \quad \mbox{for a.e.} \ \sigma \in \cM
\]
as $j \to \infty$. Now, from \eqref{eq75}, Lemma \ref{lemma4} $(i)$ and the H\"older inequality it follows
\begin{align} \label{eq79}
0&<\frac{2 \tilde{\varrho}}{C \lambda_{k+h+1}} \leq \frac{\Vert \alpha f(w_j) \Vert_{L^{r'}(\cM)}}{\Vert w_j \Vert}  \notag \\
&\leq \frac{\left( \displaystyle\int_\cM \left[ \alpha(\sigma)\left( 2 \varepsilon |w_j|+r A_1^\varepsilon |w_j|^{r-1}\right)\right]^{\frac{r}{r-1}}\right)^{\frac{r-1}{r}}}{\Vert w_j \Vert} \\
&\leq \frac{\left[ \left( 4 \varepsilon \right)^{\frac{r}{r-1}}  \Vert \alpha \Vert^{\frac{r}{r-1}}_{L^{\frac{2r}{r-1}}(\cM)}\left(\displaystyle\int_\cM |w_j|^{\frac{2r}{r-1}}\,dv_g\right)^{\frac{1}{2}}+\left( 2rA_1^\varepsilon\right)^{\frac{r}{r-1}}\displaystyle\int_\cM |w_j|^r\,dv_g\right]^{\frac{r-1}{r}}}{\Vert w_j \Vert} \notag.
\end{align}
Recalling that $\h \hookrightarrow L^s(\cM)$ is continuous for every $s \in \left[ 2,2^* \right]$ and that for every $a,b \geq 0$ and $0<p\leq 1$ we have
\(
(a+b)^p\leq a^p+b^p,
\)
we deduce from \eqref{eq79} that
\begin{equation} \label{eq80}
0<\frac{2 \tilde{\varrho}}{C \lambda_{k+h+1}} \leq \tilde{C}\left(4\varepsilon \Vert \alpha \Vert_{L^{\frac{2r}{r-1}}}+2rA_1^\varepsilon \Vert w_j \Vert^{r-2} \right)
\end{equation}
for some optimal $\tilde{C}>0$. With similar estimates it is straightforward to check that
\[
\left| \alpha(\sigma) f(w_j(\sigma))\right|^{\frac{r}{r-1}} \leq C_1^\varepsilon |w_j(\sigma)|^{\frac{2r}{r-1}}+C_2^\varepsilon |w_j(\sigma)|^r
\]
and
\[
|\alpha(\sigma)| \leq C_3^\varepsilon |w_j(\sigma)|^2+C_4^\varepsilon |w_j(\sigma)|^r
\]
choosing adequately $C_1^\varepsilon, C_2^\varepsilon, C_3^\varepsilon, C_4^\varepsilon>0$. Hence, the general Lebesgue dominated convergence Theorem  \cite[Section 4.4, Theorem 19]{MR1013117} implies
\begin{equation} \label{eq81}
\lim_{j \to \infty} \int_\cM \alpha(\sigma)F(w_j(\sigma))\, dv_g=\int_\cM \alpha(\sigma)F(w_{\infty}(\sigma))\, dv_g
\end{equation}
and
\begin{equation} \label{eq82}
\lim_{j \to \infty} \int_\cM |\alpha(\sigma) f(w_j(\sigma))|^{\frac{r}{r-1}}\, dv_g =\int_\cM |\alpha(\sigma) f(w_\infty(\sigma))|^{\frac{r}{r-1}}\,dv_g.
\end{equation}
Coupling \eqref{eq71} and \eqref{eq81}, keeping into account $(f_4)$, we see that $w_\infty=0$ is the only admissible case. At this point only two possible cases are possible. The first one is that $w_j \to 0$ in $\h$, but if that were true, letting $j \to \infty$, then we would have
\[
0<\frac{2 \tilde{\varrho}}{C \lambda_{k+h+1}} \leq 4\varepsilon \tilde{C} \Vert \alpha \Vert_{L^{\frac{2r}{r-1}}}
\]
which is impossible since $\varepsilon >0$ is arbitrary. The second one is that there exist $\eta>$ such that $\Vert w_j \Vert \geq \eta$ for each $j \in \N$. In this case, firstly we notice that that from $w_\infty=0$ and $f(0)=0$ it follows
\begin{equation} \label{eq83}
\lim_{j \to \infty} \int_\cM |\alpha(\sigma) f(w_j(\sigma))|^{\frac{r}{r-1}}\, dv_g =0.
\end{equation}
Then, thanks to \eqref{eq83}, \eqref{eq75} becomes
\[
0<\frac{2 \tilde{\varrho \eta}}{C \lambda_{k+h+1}} \leq 0,
\]
which is clearly a contradiction.
\end{proof}
In the sequel, given a closed subspace $Y$ of $\h$ we will denote with
\(
P_{Y}\colon \h \to Y
\)
the usual orthogonal projection.
\begin{proposition} \label{prop4}
Suppose $f$ satisfies $(f_1)-(f_4)$, $\lambda \in \R$ and let $(w_j)_j \subset \h$ be a sequence such that
\begin{equation} \label{eq84}
\left( J_{\lambda}(w_j)\right) \quad \mbox{is bounded}
\end{equation}
\begin{equation} \label{eq85}
P_{X_2}w_j \to 0 \quad \mbox{in} \ \h
\end{equation}
\begin{equation} \label{eq86}
P_{X_1 \oplus X_3} \nabla J_{\lambda}(w_j) \to 0 \quad \mbox{in} \ \h.
\end{equation}
Then $(w_j)_j$ is bounded in $\h$.
\end{proposition}
\begin{proof}
We argue by contradiction, and we suppose that
\begin{equation} \label{eq87}
\Vert w_j \Vert \to \infty
\end{equation}
as $j \to \infty$. Normalizing we assume up to a subsequence
\[
\frac{w_j}{\Vert w_j \Vert} \rightharpoonup w_\infty \quad \mbox{in} \ \h
\]
and
\begin{equation} \label{eq88}
\frac{w_j}{\Vert w_j \Vert} \to w_\infty \quad \mbox{in} \ L^s(\cM)
\end{equation}
as $j \to \infty$ for all $s \in \left[2,2^*\right)$.

Clearly, we can write
\begin{equation}\label{eq89}
w_j=P_{X_2}w_j+P_{X_1\oplus X_3}
\end{equation}
with $P_{X_2}w_j \to 0$. Recalling \eqref{eq51}, \eqref{eq52} and\eqref{eq89} we have
\begin{align} \label{eq90}
\langle P_{X_1 \oplus X_3} \nabla J_\lambda (w_j), w_j \rangle&=\langle \nabla J_\lambda (w_j), w_j \rangle-\langle P_{X_2} \nabla J_\lambda (w_j), w_j \rangle \notag \\
&= \Vert w_j \Vert^2-\lambda \Vert w_j \Vert^2_{L^2(\cM)}-\int_\cM \alpha(\sigma) f(w_j(\sigma)) w_j(\sigma) \, dv_g \\
&-\langle P_{X_2} \left( w_j-\Delta_g^{-1} \left( \lambda w_j + \alpha f(w_j)\right)\right),w_j \rangle \notag
\end{align}
By orthogonality we get
\begin{displaymath}
\langle P_{X_2}w, v \rangle =\langle P_{X_2} w, P_{X_1+\oplus X_3} v+P_{X_2}v \rangle=  \langle P_{X_2} w, P_{X_2}v \rangle
\end{displaymath}
and
\[
\langle w,P_{X_2}v \rangle =\langle  P_{X_1+\oplus X_3} w+P_{X_2},P_{X_2} v \rangle= \langle P_{X_2} w, P_{X_2}v \rangle
\]
for every $w,v \in \h$, which means that $P_{X_2}$ is a symmetric operator. In virtue of that, we have
\begin{align} \label{eq91}
\langle P_{X_2} \left( w_j-\Delta_g^{-1} \left( \lambda w_j + \alpha f(w_j)\right)\right),w_j \rangle & = \Vert P_{X_2}w_j \Vert^2-\lambda \langle \Delta^{-1}_g w_j,P_{X_2}w_j \rangle \notag \\
&-\langle \Delta_g^{-1}\left( \alpha f(w_j) \right),P_{X_2} w_j \rangle.
\end{align}
Recalling \eqref{eq50} we get
\begin{multline} \label{eq92}
\lambda \langle P_{X_2}w_j,\Delta_g^{-1} w_j\rangle+\langle P_{X_2}w_j, \Delta_g^{-1} \left( \alpha f(w_j) \right) \rangle \\
= \lambda \Vert P_{X_2} w_j \Vert_{L^2(\cM)}^2+\int_\cM \alpha(\sigma) f(w_j(\sigma))P_{X_2}w_j(\sigma)\,dv_g
\end{multline}
Inserting \eqref{eq91} and \eqref{eq92} in \eqref{eq90} we obtain
\begin{align} \label{eq93}
\langle P_{X_1 \oplus X_3} \nabla J_\lambda (w_j),w_j \rangle &= 2 J_\lambda (w_j)+2\int_\cM \alpha(\sigma)F(w_j(\sigma))\, dv_g \notag \\
&- \Vert P_{X_2} w_j \Vert^2+\lambda \Vert P_{X_2}  w_j\Vert^2_{L^2(\cM)} \notag \\
& {} +\int_\cM \alpha(\sigma)f(w_j(\sigma))P_{X_2}w_j(\sigma)\, dv_g.
\end{align}
Reordering the terms in \eqref{eq93} and using \eqref{eq84}, \eqref{eq85}, \eqref{eq86} and \eqref{eq87} we get
\begin{multline} \label{eq94}
\frac{1}{\Vert w_j \Vert^r} \left( 2 \int_\cM \alpha(\sigma)F(w_j(\sigma))\,dv_g -\int_\cM \alpha(\sigma)f(w_j(\sigma)) w_j(\sigma)\, dv_g \right. \\
\left. {} +\int_\cM \alpha(\sigma)f(w_j(\sigma))P_{X_2}w_j \, dv_g \right) \to 0
\end{multline}
as $j \to \infty$.

\textit{Claim}: $w_\infty=0$

We first need  to show
\begin{equation} \label{eq95}
\frac{\displaystyle\int_\cM \alpha(\sigma)f(w_j(\sigma))P_{X_2}w_j \, dv_g}{\Vert w_j \Vert^r} \to 0
\end{equation}
as $j \to \infty$. Indeed, having $X_2$ finite dimension, all norms are equivalent. Moreover, all eigenfunctions are bounded by \cite[Theorem 3.1]{MR4020749}, thus
\[
\Vert P_{X_2} w_j \Vert_{L^\infty(\cM)} \to 0
\]
as $j \to \infty$ remembering \eqref{eq85}. Then, from Lemma \ref{lemma4} $(i)$
\begin{multline*}
\left| \frac{\displaystyle\int_\cM \alpha (\sigma) f(w_j(\sigma))P_{X_2}w_j(\sigma)\,dv_g}{\Vert w_j \Vert^r} \right| \\ \leq \frac{2\varepsilon \displaystyle\int_\cM \alpha(\sigma)w_j(\sigma)\,dv_g+rA_1^\varepsilon \Vert P_{X_2} w_j \Vert_{L^{\infty}(\cM)}\displaystyle\int_\cM\alpha(\sigma)|w_j(\sigma)|^{r-1}\, dv_g}{\Vert w_j \Vert^r}.
\end{multline*}
Applying the H\"older inequality twice and recalling $\h \hookrightarrow L^2(\cM)$ it follows
\begin{multline*}
\left| \frac{\displaystyle\int_\cM \alpha (\sigma) f(w_j(\sigma))P_{X_2}w_j(\sigma)\,dv_g}{\Vert w_j \Vert^r} \right| \\
\leq \frac{2\varepsilon C \Vert \alpha \Vert_{L^2(\cM)}}{\Vert w_j \Vert^{r-2}}+\frac{rA_1^\varepsilon \Vert P_{X_2} w_j \Vert_{L^{\infty}(\cM)} \Vert \alpha \Vert^r_{L^{r}(\cM)}\Big\Vert \frac{w_j}{\Vert w_j \Vert} \Big\Vert^{r-1}_{L^r(\cM)}}{\Vert w_j \Vert}
\end{multline*}
for some $C>0$. Now the validity of \eqref{eq95} follows from the boundedness of the sequence $w_j / \Vert w_j \Vert$ in $L^r(\cM)$. In virtue of \eqref{eq95}, combining \eqref{eq84} with $(f_4)$, we obtain
\begin{multline} \label{eq100}
o(1) =  \frac{2\displaystyle\int_\cM \alpha(\sigma)F(w_j(\sigma))\,dv_g-\displaystyle\int_\cM \alpha(\sigma)f(w_j(\sigma))w_j(\sigma)\,dv_g}{\Vert w_j \Vert^r} \\ \leq \frac{(2-r)\displaystyle\int_\cM \alpha(\sigma)F(w_j(\sigma))\,dv_g}{\Vert w_j \Vert^r}\leq 0
\end{multline}
from which we deduce
\[
\lim_{j \to \infty} \frac{\displaystyle\int_\cM \alpha(\sigma)F(w_j(\sigma))\,dv_g}{\Vert w_j \Vert^r} =0
\]
To conclude the proof of the claim, it suffices to apply the previous result to
\[
\left\Vert \frac{w_j}{\Vert w_j \Vert} \right\Vert_{L^r(\cM)} \leq \frac{A_4 \Vert \alpha \Vert_{L^1(\cM)}}{A_3\Vert w_j \Vert^r}+\frac{1}{A_3 \Vert w_j \Vert^r}\displaystyle\int_\cM \alpha(\sigma)F(w_j(\sigma))\,dv_g
\]
where we used Lemma \ref{lemma4} $(iii)$. Now, we observe that
\[
0\leftarrow \frac{J_\lambda(w_j)}{\Vert w_j \Vert^2}=\frac{1}{2}-\frac{\lambda}{2}\Bigg\Vert \frac{w_j}{\Vert w_j \Vert} \Bigg\Vert_{L^2(\cM)}^2-\frac{1}{\Vert w_j \Vert^2}\displaystyle\int_\cM \alpha(\sigma) F(w_j(\sigma))\, dv_g.
\]
Recalling $w_j/\Vert w_j \Vert \to 0$ in $L^2(\cM)$ we obtain
\begin{equation} \label{eq96}
\frac{1}{\Vert w_j \Vert^2}\displaystyle\int_\cM \alpha(\sigma) F(w_j(\sigma))\, dv_g \to \frac{1}{2}
\end{equation}
as $j \to \infty$.  Furthermore, from Lemma \ref{lemma4} $(iii)$ it follows
\begin{equation} \label{eq97}
\frac{1}{\Vert w_j \Vert^2}\displaystyle\int_\cM \alpha (\sigma) |w_j(\sigma)|^r\, dv_g \leq \frac{A_4 \Vert \alpha \Vert_{L^1(\cM)}}{A_3 \Vert w_j \Vert^2}+\frac{1}{A_3\Vert w_j \Vert^2}\displaystyle\int_\cM \alpha(\sigma) F(w_j(\sigma))\, dv_g.
\end{equation}
Because of \eqref{eq96}, the second member of \eqref{eq97} is bounded and so there exist  a $\tilde{C}>0$ such that
\begin{equation} \label{eq98}
\displaystyle\int_\cM \alpha(\sigma) |w_j(\sigma)|^r\, dv_g \leq \tilde{C}\Vert w_j \Vert^2.
\end{equation}
At this point, applying Lemma \ref{lemma4} $(ii)$, the H\"older inequality and \eqref{eq98}, we notice
\begin{gather*}
\frac{\displaystyle\int_\cM \left|  \alpha (\sigma) f(w_j(\sigma))P_{X_2}w_j(\sigma)\right|\,dv_g}{\Vert w_j \Vert^r} \\
\leq \frac{\Vert P_{X_2}w_j\Vert_{L^\infty (\cM)}}{\Vert w_j \Vert^2} \left(A_2 \Vert \alpha \Vert_{L^1(\cM)}
+ A_2^\varepsilon\displaystyle\int_\cM |\alpha(\sigma)|^{\frac{1}{r}}|\alpha(\sigma)|^{\frac{r-1}{r}}|w_j(\sigma)|^{r-1} \right)  \\
\leq \Vert P_{X_2} w_j \Vert_{L^\infty} \left[ \frac{A_2 \Vert \alpha \Vert_{L^1(\cM)}}{\Vert w_j \Vert^2}  + \frac{A_2^\varepsilon \Vert \alpha \Vert^{\frac{1}{r}}_{L^1(\cM)}}{\Vert w_j \Vert^{\frac{2}{r}}} \left( \frac{\Vert \alpha w_j \Vert^r_{L^r(\cM)}}{\Vert w_j \Vert^2}\right)^{\frac{r-1}{r}}\right]  \\
\leq \Vert P_{X_2} w_j \Vert_{L^\infty} \left[ \frac{A_2 \Vert \alpha \Vert_{L^1(\cM)}}{\Vert w_j \Vert^2}  +\frac{A_2^\varepsilon \tilde{C}^{1-\frac{1}{r}}\Vert \alpha \Vert^{\frac{1}{r}}_{L^1(\cM)}}{\Vert w_j \Vert^{\frac{2}{r}}} \right] ,
\end{gather*}
which implies
\begin{equation} \label{eq99}
\lim_{j \to \infty} \frac{\displaystyle\int_\cM \left|  \alpha (\sigma) f(w_j(\sigma))P_{X_2}w_j(\sigma)\right|\,dv_g}{\Vert w_j \Vert^r}=0.
\end{equation}
Dividing \eqref{eq93} by $\Vert w_j \Vert^2$ and using \eqref{eq84}, \eqref{eq85}, \eqref{eq86} and \eqref{eq99} we get
\[
\frac{1}{\Vert w_j \Vert^2} \left( \displaystyle\int_\cM \alpha(\sigma) F(w_j(\sigma))\,dv_g-\displaystyle\int_\cM \alpha(\sigma) f(w_j(\sigma))w_j(\sigma)\, dv_g\right) \to 0
\]
as $j \to \infty$. To conclude the proof, we argue as did in \eqref{eq100} to obtain
\begin{equation} \label{eq101}
\lim_{j \to \infty} \frac{1}{2}\displaystyle\int_\cM \alpha(\sigma)F(w_j(\sigma))\, dv_g=0.
\end{equation}
Clearly \eqref{eq96} and \eqref{eq101} are not compatible.
\end{proof}
\begin{proposition} \label{prop 5}
Assume $f$ satisfies $(f_1)-(f_4)$. For any $\varrho>0$ there exists $\eta_\varrho>0$ such that for any $\eta', \eta'' \in (0,\eta_\varrho)$, with $\eta' <\eta''$ we have that $\nabla \left(J_\lambda, X_1 \oplus X_3, \eta', \eta'' \right)$ is verified for all $\lambda \in (\lambda_{k-1}+\varrho, \lambda_{k+h+1}-\varrho)$.
\end{proposition}
\begin{proof}
By contradiction we suppose there is $\tilde{\varrho}>0$ such that for any $\eta_{\tilde{\varrho}}>0$ we can find $\tilde{\lambda} \in \left[ \lambda_{k-1}+\tilde{\varrho},\lambda_{k+h+1}-\tilde{\varrho} \right)$ and $\eta' <\eta''$ such that
\[
(\nabla) \left( J_\lambda,X_1\oplus X_3, \eta', \eta'' \right)
\]
does not hold. If so, it is possible to find a sequence $(w_j)_j \subset \h$ such that
\[
J_{\tilde{\lambda}}(w_j) \in \left[ \eta', \eta'' \right]
\]
\begin{equation} \label{eq103}
\dist (w_j,X_1\oplus X_3) \to 0 \quad \mbox{as} \ j \to \infty
\end{equation}
\begin{equation} \label{eq104}
P_{X_1 \oplus X_3} \nabla J_{\tilde{\lambda}}(w_j) \to 0 \quad \mbox{as} \ j \to \infty.
\end{equation}
Because of that, Proposition \ref{prop4} can be applied, thus $(w_j)_j$ is bounded in $\h$. Hence, up to a subsequence,
\begin{equation} \label{eq106}
w_j \rightharpoonup w_\infty \quad \mbox{in} \ \h
\end{equation}
\begin{equation} \label{eq107}
w_j \to w_\infty \quad \mbox{in} \ L^s(\cM) \quad \mbox{for all} \ s \in \left[2, 2^* \right)
\end{equation}
\[
w_j(\sigma) \to w_{\infty} \quad \mbox{a.e in} \cM
\]
as $ j \to \infty$. Now, arguing as we did to obtain \eqref{eq77}, we can find $\tilde{A}_1^\varepsilon, \tilde{A}_2^\varepsilon >0$ such that
\[
\displaystyle\int_\cM | \alpha(\sigma)f(w_j(\sigma))|^{\frac{r}{r-1}}\, dv_g \leq \tilde{A}_1^\varepsilon + A_2^{\varepsilon} \displaystyle\int_\cM |w_j(\sigma)|^r \, dv_g.
\]
Since $ w_j \to w_\infty$ in $L^r(\cM)$ there is $\tilde{C}>0$ such that
\[
\displaystyle\int_\cM | \alpha(\sigma)f(w_j(\sigma))|^{\frac{r}{r-1}}\, dv_g \leq \tilde{C}.
\]
Then, recalling that $\Delta_g^{-1}$ is a compact operator,
\begin{equation} \label{eq102}
P_{X_1\oplus X_3} \Delta_g^{-1} \left( \tilde{\lambda}w_j+\alpha f(w_j) \right) \to P_{X_1\oplus X_3} \Delta_g^{-1} \left( \tilde{\lambda}w_\infty+\alpha f(w_\infty) \right).
\end{equation}
Recalling \eqref{eq52}, we have
\[
P_{X_1 \oplus X_3} \nabla J_{\lambda}(w_j)=w_j-P_{X_2}w_j-P_{X_1 \oplus X_3} \Delta_g^{-1} \left( \tilde{\lambda} w_j + \alpha f(w_j) \right).
\]
Since that, \eqref{eq102}, \eqref{eq103} and \eqref{eq104} we deduce
\[
w_j \to P_{X_1 \oplus X_3} \Delta_g^{-1} \left( \tilde{\lambda} w_\infty + \alpha f(w_\infty) \right)
\]
in $\h$ as $j \to \infty$. Now, on one hand from \eqref{eq51} and\eqref{eq104} it follows
\begin{equation} \label{eq105}
\langle J_{\tilde{\lambda}}(w_j), \varphi \rangle = \langle w_j, \varphi \rangle -\tilde{\lambda} \langle w_j, \varphi \rangle_{L^2(\cM)}-\displaystyle\int_\cM \alpha (\sigma)f(w_j(\sigma)) \varphi (\sigma) \, dv_g \to 0
\end{equation}
for any $\varphi \in X_1 \oplus X_3$ as $j \to \infty$.
\end{proof}
On the other hand, from \eqref{eq106} and \eqref{eq107} we also have
\begin{equation} \label{eq108}
\langle J_{\tilde{\lambda}}, \varphi \rangle \to \langle w_\infty, \varphi \rangle -\tilde{\lambda} \langle w_\infty, \varphi \rangle_{L^2(\cM)}-\displaystyle\int_\cM \alpha(\sigma) f(w_j(\sigma)) \varphi (\sigma) \, dv_g
\end{equation}
for any $\varphi \in X_1 \oplus X_3$. Coupling \eqref{eq105} and \eqref{eq108} we get that $w_\infty$ is a critical point for $J_{\tilde{\lambda}}$ constrained on $X_1 \oplus X_3$. Then, we can apply Proposition \ref{prop3} to obtain $w_\infty =0$. But, since $J_{\tilde{\lambda}}(w_j) \geq \eta'$, $w_j \to w_\infty $ in $\h$, exploiting the continuity of $J_{\tilde{\lambda}}$ we obtain $J_{\tilde{\lambda}}(w_\infty)>0$. This is a contradiction being $J_{\tilde{\lambda}}(0)=0$.
\section{Proof of  Theorem \ref{th2}}

We begin with a technical result.
\begin{lemma} \label{lemma5}
If $f$ verifies $(f_1)-(f_4)$ then
\[
\lim_{\lambda \to \lambda_k} \sup_{w \in E_{k+h}} J_\lambda(w)=0
\]
\end{lemma}
\begin{proof}
We start noticing that from Lemma \ref{lemma4} $(iii)$ it follows
\[
\lim_{\xi \to \pm \infty} J_{\lambda} (\xi w)=-\infty
\]
for all $w \in E_{k+h}$, thus
\[
\sup_{w \in E_{k+h}} J_\lambda(w)
\]
is achieved. Now, by contradiction we suppose there is a sequence $\tau_j \to \lambda_k$ as $j \to \infty$ and a sequence $(w_j)_j \subset E_{k+h}$ such that
\begin{equation} \label{eq109}
J_{\tau_j}(w_j)=\sup_{w \in E_{k+h}} J_\lambda(w)>\gamma
\end{equation}
for some $\gamma >0$. We split the proof analysing separately the case $(w_j)_j$ bounded and unbounded. In the first one, since the weak and the strong topology coincide, we can suppose $w_j \to w_\infty$ in $E_{k+h}$. In order to reach a contradiction, keeping into account \eqref{eq109} and letting $j \to \infty$, it suffices to apply Lemma \ref{lemma3} to obtain
\[
\gamma \leq J_{\lambda_k} (w_\infty)= \left( \lambda_{k+h}-\lambda_k \right) - \displaystyle\int_\cM \alpha(\sigma) F(w_{\infty}(\sigma))\, dv_g \leq 0.
\]
Instead, if $(w_j)_j$ is unbounded, we can assume $\Vert w_j \Vert \to \infty$ as $j \to \infty$. From Lemma \ref{lemma4} $(iii)$ it follows
\[
0<\gamma \leq J_{\tau_j}(w_j) \leq \frac{1}{2} \Vert w_j \Vert^2-\frac{\tau_j}{2} \Vert w_j \Vert^2_{L^2(\cM)}-A_3 \vert w_j \Vert^r_{L^r(\cM)}+A_4 \Vert \alpha \Vert_{L^1(\cM)}.
\]
Exploiting again the fact that on the finite-dimensional subspace $E_{h+k}$ all norms are equivalent, the right hand side of the above inequality goes to $- \infty$ concluding the proof.
\end{proof}
\begin{proof}[Proof of Theorem \ref{th2}]
We want to apply \cite[Theorem 2.10]{MR1655535}. We start choosing $ \varrho>0$. In correspondence of that, thanks to Proposition \ref{prop 5} there are $\eta_\varrho, \eta', \eta''>0$, with $\eta'<\eta''<\eta_\varrho$ such that $\nabla \left(J_\lambda, X_1 \oplus X_3, \eta', \eta'' \right)$ is verified for all $\lambda \in (\lambda_{k-1}+\varrho, \lambda_{k+h+1}-\varrho)$. Exploiting Lemma \ref{lemma5} we also have the existence of $\overline{\varrho}>0$, with $\overline{\varrho} \leq \varrho$ such that
\[
\sup_{w \in E_{k+h}} J_\lambda(w) \leq \eta'
\]
for $\lambda \in (\lambda_k+\overline{\varrho}, \lambda_k)$. At this point, recalling Propositions \ref{prop1} and \ref{prop2}, all hypothesis of Theorem 2.10 in \cite{MR1655535} are satisfied, and we have the existence of two non trivial critical points $w_1$ and $w_2$ such that
\[
J_\lambda(w_i) \in \left[ \eta', \eta'' \right] \quad (i=1,2).
\]
The third critical point $w_3$ is a consequence of the classical Linking Theorem. Furthermore, from Lemma \ref{lemma5}, choosing $\lambda$ sufficiently close to $\lambda_k$, we can see that
\[
J_\lambda (w_i) < \sup_{w \in E_{k+h}}J_\lambda (w) \leq J_\lambda (w_3), \quad (i=1,2)
\]
proving that $w_1$, $w_2$, $w_3$ are distinct.
\end{proof}
\bibliographystyle{amsplain}
\bibliography{TSM}
\end{document}